\documentclass[12pt,leqno,amscd,amssymb,verbatim, url]{amsart}
\usepackage{amsfonts,amssymb,graphicx}
\usepackage{amsmath,amscd}
\usepackage{hyperref}
\usepackage{mathrsfs}
\let\oldsection=

\renewcommand{\subsection}[1]{\par\vspace{.18in}\noindent\addtocounter{subsection}{1}\setcounter{equation}{0}{\bf\thesubsection\hspace{9pt}#1}}

\usepackage[english]{babel}
\usepackage[latin1]{inputenc}
\usepackage{times}
\usepackage[T1]{fontenc}
\usepackage{amsthm, amsmath, amssymb}
\usepackage{tikz}

\usetikzlibrary{arrows,shapes}
\usetikzlibrary{decorations.pathmorphing}
\usetikzlibrary{calc}

\newtheorem{thm}{Theorem}[subsection]

\makeatletter\let\c@fact\c@theorem\makeatother\newtheorem{lem}[thm]{Lemma}
\newtheorem{cor}[thm]{Corollary}

\newtheorem{prop}[thm]{Proposition}

\theoremstyle{definition}

\newtheorem{defn}[thm]{Definition}
\newtheorem{prob}[thm]{Problem}

\newtheorem{ques}[thm]{Questions}
\newtheorem{rem}[thm]{Remark}
\newtheorem{rems}[thm]{Remarks}

\numberwithin{equation}{subsection}
\numberwithin{thm}{section}
\newcommand{\Smod}{A{\text{\rm --mod}}}

\newcommand{\Agrmod}{{A{\text{\rm --grmod}}}}
\newcommand{\Bgrmod}{{B{\text{\rm --grmod}}}}

\newcommand{\grhom}{\text{\rm hom}}

\newcommand{\grHom}{\text{\rm hom}}
\newcommand{\grExt}{\text{\rm ext}}

\newcommand{\wgr}{\widetilde{\text{\rm gr}}\,}

\newcommand{\op}{{\text{\rm op}}}

\newcommand{\wrad}{\widetilde{\text{\rm rad}}\,}

\newcommand{\Lamreg}{\Lambda^+_{\text{\rm reg}}}

\newcommand{\gr}{\text{\rm gr}}

\newcommand{\sZ}{{\mathscr{Z}}}

\newcommand{\sC}{{\mathscr {C}}}

\newcommand{\Ext}{{\text{\rm Ext}}}

\newcommand{\modR}{\mbox{mod--}R}

\newcommand{\BJmod}{B/J\mbox{--mod}}
\newcommand{\BJgrmod}{B/J\mbox{--grmod}}
\newcommand{\Amod}{A\mbox{--mod}}
\newcommand{\Bmod}{B\mbox{--mod}}
\newcommand{\Hom}{\text{\rm Hom}}

\newcommand{\End}{\operatorname{End}}

\newcommand{\ext}{\operatorname{Ext}}
\renewcommand{\hom}{\operatorname{hom}}

\newcommand{\ind}{\operatorname{ind}}
\newcommand{\sE}{\operatorname{{\mathscr E}}}

\newcommand{\sO}{{\mathscr{O}}}

\newcommand{\rad}{\operatorname{rad}}

\newcommand{\rDelta}{\Delta^{\text{\rm red}}}
\newcommand{\rnabla}{\nabla_{\text{\rm red}}}
\newcommand{\wM}{{\widetilde{M}}}
\newcommand{\wA}{{\widetilde{A}}}
\newcommand{\wB}{{\widetilde{B}}}

\newcommand{\wDelta}{{\widetilde{\Delta}}}

\newcommand{\wN}{{\widetilde{N}}}

\newcommand{\wP}{{\widetilde{P}}}

\newcommand{\wS}{{\widetilde{S}}}

\newcommand{\wL}{{\widetilde{L}}}

\newcommand{\wY}{{\widetilde{Y}}}

\newcommand{\blist}{\begin{list}{\rom{(\roman{enumi})}}{\setlength
{\leftmarg in}{0em} \setlength{\itemindent}{7ex}
\setlength{\labelsep}{2ex}\setlength{\listparindent}{\parindent}
\usecounter{enumi}}}
\newcommand{\elist}{\end{list}}

\begin{document}
\dedicatory{We dedicate this paper to the memory of J.A. Green}
\begin{abstract} In previous work, the authors introduced the notion of Q-Koszul algebras, as a tool to ``model"  module categories for semisimple algebraic groups over fields of large characteristics. 
Here we suggest the model extends to small characteristics as well. In particular, we present several conjectures
in the modular representation theory of semisimple groups which these algebras inspire.  They provide a new world-view of modular representation theory, potentially valid for
some root systems in all characteristics.  In fact, we give a non-trivial example in which $p=2$. This paper begins a
systematic study of Q-Koszul algebras, viewed as  interesting objects in their own right.  \end{abstract}

 \title[Q-Koszul algebras and three conjectures]{Q-Koszul algebras and three conjectures}\author{Brian J. Parshall}
\address{Department of Mathematics \\
University of Virginia\\
Charlottesville, VA 22903} \email{bjp8w@virginia.edu {\text{\rm
(Parshall)}}}
\author{Leonard L. Scott}
\address{Department of Mathematics \\
University of Virginia\\
Charlottesville, VA 22903} \email{lls2l@virginia.edu {\text{\rm
(Scott)}}}

\thanks{Research supported in part by the National Science
Foundation}
\maketitle

\section{Introduction} Let $A=\bigoplus_{i\geq 0}A_i$ be a positively graded algebra over a
field $k$. For simplicity, assume $A$ is finite dimensional and that $k$ is algebraically closed.
Let $\Agrmod$ be the category of finite dimensional $\mathbb Z$-graded $A$-modules, and let $\Amod$
be the category of finite dimensional $A$-modules. The abelian categories $\Agrmod$ and $\Amod$ each have enough
projective (and injective) modules. If $M=\bigoplus_iM_i\in\Agrmod$, for any integer $r$, $M\langle r\rangle\in \Agrmod$
is defined by $M\langle r\rangle_i:=M_{i-r}$. If $\grExt^\bullet$ denotes the Ext-bifunctor in $\Agrmod$, then, for $M,N\in\Agrmod$, 
\begin{equation}\label{gradedExt}\Ext^n_A(M,N)=\bigoplus_{r
\in \mathbb Z}\grExt^n_A(M,N\langle r\rangle)), \quad\forall n\in\mathbb N,\end{equation}
where the left hand side is computed in $\Amod$, after forgetting the gradings on $M$ and $N$. 
 One
says that $A$ is Koszul provided that each irreducible module $L$, when
regarded as a graded $A$-module concentrated in grade 0, has a projective resolution $P^\bullet\twoheadrightarrow L$ in $\Agrmod$ in which $P^n$ has head which is pure of grade $n$. Equivalently, $\grExt^n_A(L,L'\langle r\rangle)\not=0\implies n=r$ for any two irreducible $A$-modules  $L,L'$
concentrated in grade 0.

Ever since the pioneering work in \cite{BGS} (see also \cite{PS0}), Koszul algebras have played a prominent role
in representation theory. For example, \cite{BGS} proved that if $\sO_0$ denotes
the principal block for the category $\sO$ of a complex semisimple Lie algebra $\mathfrak g$, then 
$\sO_0$ is equivalent to the module category of a finite dimensional Koszul algebra
$A$. Also, \cite{AJS} and \cite{Riche} show the restricted Lie algebra of a semisimple, simply connected algebraic group $G$ in
characteristic $p>0$ is Koszul, provided that $p$ is sufficiently large, depending on the root system. Nevertheless,
the Koszul property generally fails for irreducible $G$-modules outside the Janzten region (for any
$p$). More precisely, the finite dimensional algebras $A_\Gamma$ below governing the representation theory of $G$ are
mostly not Koszul.

Let
$\Gamma$ be a finite set of dominant weights on $G$ which is an ideal in the dominance order. The
category of finite dimensional rational $G$-modules which have composition factors consisting of those irreducible
modules $L(\gamma)$, $\gamma\in \Gamma$, is equivalent to the module category $A_\Gamma$-mod,
where $A_\Gamma$ is a finite dimensional algebra $A_\Gamma$. While the algebra $A_\Gamma$ is necessarily
quasi-hereditary, that fact alone is not, without further structure, sufficient to understand more deeply the
representation theory of $G$. In this spirit, the recent paper \cite{PS13a} introduced the notion of a
``standard Q-Koszul algebra" as a potential model for the representation theory of $G$---albeit (at that time) for large primes and
$p$-regular weights.  In more
detail, assume that $\Gamma$ is restricted to consist of $p$-regular weights. (Thus, $\Gamma$ is a finite
ideal in the poset of all $p$-regular dominant weights.) The algebra $A_\Gamma$
has a filtration by ideals (which arise from the radical series of its quantum analogue), and one can
form the associated graded algebra $\wgr A_\Gamma$.  Then a main result in \cite{PS13a} proves that
the algebra $\wgr A_\Gamma$ is a standard Q-Koszul algebra, provided that $p$ is ``sufficiently large."  Here ``sufficiently large"
means that the Lusztig character formula is assumed to hold for $G$, and $p\geq 2h-2$ is odd (with $h$ the
Coxeter number of $G$). Further, much of the homological
algebra of $A_\Gamma$ can be determined from that of $\wgr A_\Gamma$, but now assuming only that
$p\geq 2h-2$ is odd. (See the remarks after Conjecture IIa in \S5 below.)

 A major aim of this paper, undertaken in its final sections, is to suggest a much broader role for Q-Koszul
 algebras in the representation theory of $G$, one valid also in smaller characteristics and for singular weights. 
 Earlier sections initiate a systematic study of Q-Koszul algebras, giving complete definitions and establishing
 some basic (but new) general results, not even yet observed in the large prime cases studied earlier.

In more detail, a finite dimensional, positively graded algebra $A$ is Q-Koszul provided that the grade 0 subalgebra $A_0$ is quasi-hereditary (with weight poset
denoted $\Lambda$). In addition, it
is required that 
\begin{equation}\label{QKoszul}\grExt^n_A(\Delta^0(\lambda),\nabla_0(\mu)\langle r\rangle)\not=0\implies n=r,\,\, \forall
\lambda,\mu\in\Lambda, n\in{\mathbb N}, r\in{\mathbb Z}\end{equation}
 Here $\Delta^0(\lambda)$ (respectively, $\nabla_0(\mu)$) is the standard
(respectively, costandard) module of corresponding to $\lambda$ (respectively, $\mu$) in $\Lambda$. Thus, if $A_0$ is semisimple, $A$ is
just a Koszul algebra. But in the situations we have in mind, $A_0$ is hardly ever semisimple!  Q-Koszul
algebras are studied in \S2. One main result, given in Theorem \ref{firstprop}, is  that $A$ is tight (i.~e., $A$ is generated
by $A_0$ and $A_1$). Then Theorem \ref{thirdprop} shows that $A$ is a quadratic algebra (see \S2 for a precise definition). This suggests the (future) project of
explicitly describing $A$ by generators and relations (for particular Q-Koszul algebras of interest
in modular representation theory). 

Another important result, also given in Theorem \ref{thirdprop}, shows that if $A$ is Q-Koszul, then the left $A_0$-module
$A_1$ has a $\Delta^0$-filtration---in fact, the later Theorem \ref{tensortensor} shows that the $(A_0,A_0^{\text{\rm op}})$-bimodule
$A_1$ has a $\Delta^0\otimes_k\Delta^{0,\text{\rm op}}$-filtration. For $s\geq 1$ and $r\geq 0$, let $\Omega_r(A_s)$
be the $r$th syzygy module of $A_r$. Then Lemma \ref{biglemma}  shows that the $A_0$modules $\Omega_{s-1}(A_s),
\Omega_s(A_s), \Omega_{s+1}(A_s), \cdots$ all have $\Delta^0$-filtrations. Of course, this extends the result just
mentioned from Theorem \ref{thirdprop} since $\Omega_0(A_1)=A_1$. \S2 contains a number of similar results,
often cast in the more general setting of $n$-Q-Koszul algebras, which are sometimes assumed to be quasi-hereditary
(automatic in the standard Q-Koszul case).

Section 3 is concerned with standard Q-Koszul algebras. Suppose that a finite dimensional algebra $A$ is positively graded and quasi-hereditary (with weight poset 
$\Lambda$).  For $\lambda\in\Lambda$, let $\Delta(\lambda)$ and $\nabla(\lambda)$ denote the corresponding
standard and costandard modules, respectively. It is known from \cite{PS10} that the subalgebra $A_0$ of pure
grade 0 is also quasi-hereditary with weight poset $\Lambda$ (and standard and costandard modules denoted
$\Delta^0(\lambda)$ and $\nabla_0(\lambda)$, respectively, for $\lambda\in\Lambda$). We say that $A$
is a {\it standard Q-Koszul algebra} provided that
\begin{equation}\label{standard}\begin{cases} \grExt^n_A(\Delta(\lambda),\nabla_0(\mu)\langle r\rangle)\not=0\implies
n=r;
\\
\grExt^n(\Delta^0(\mu),\nabla(\lambda)\langle r\rangle)\not=0\implies n=r.\end{cases}\quad\forall \lambda,\mu\in\Lambda,
n\in{\mathbb N}, r\in{\mathbb Z}.\end{equation}
  The main result, given in Theorem \ref{better}, proves that if $A$ is 
standard Q-Koszul, then it is Q-Koszul. Interestingly, the result is understood conceptually from
a triangulated category point of view, inspired by similar methods in \cite{CPS1}.  Another result in this section,
Corollary \ref{cor3.5}, draws on the work of \S2 to show that the grade 2 relation module $W_2$ of a standard Q-Koszul algebra $A$ has an
especially nice $\Delta^0$-filtration. This happens in spite of the fact that its grade 2 term $A_2$ need not have a $\Delta^0$-filtration.
That this can occur is a consequence of an example  in a Weyl module context, due to Will Turner, and discussed in \S5.
 
 Both sections 5 and 6 treats a highly non-trivial case in which the characteristic $p$ is small. In fact, $p=2$.  Explicitly, we consider the
Schur algebra $S(5,5)$ associated to 5-homogenous polynomial representations of $GL_5(k)$ when $k$ has characteristic 2. In this case,  {\it all} the weights are $2$-singular, and there is no
proposed analog of the Lusztig character formula for irreducible modules. Nevertheless, we prove that $\wgr S(5,5)$ is standard
Q-Koszul. (Actually, we focus on $\wgr A$ for $A$ the ``principal block" of $S(5,5)$, leaving details beyond this case to the reader. By the principal block we mean the block containing the determinant representation.)
This result takes as its starting point computer calculations done by Jon Carlson
\cite{Car}.  It is interesting to note that while the Schur algebra $S(5,5)$ is quasi-hereditary, the graded
algebra $\gr S(5,5)$, obtained (unlike $\wgr  S(5,5)$) from the radical series filtration of $S(5,5)$ itself, is not quasi-hereditary. In addition, $S(5,5)$ is not Koszul (nor is the graded algebra $\gr S(5,5)$ Koszul either). See Remark \ref{rem4.2}.  

 Section 7 discusses three natural conjectures suggested by this paper in combination with our previous work. Conjecture II proposes a generalization to small primes and singular weights of (already interesting) homological results in the large prime, $p$-regular cases. This conjecture does not involve graded algebras in its statement. However, it is inspired by Q-Koszul theory, which might well play a role in its proof. Conjecture I asserts that a rich supply
of Q-Koszul algebras is available, while two supplementary conjectures, labeled Conjectures IIa and IIb, show the relevance of these algebras to Conjecture II. Finally, Conjecture III, motivated by Koszul algebra theory in the quantum
case, provides calculations, in terms of Kazhdan-Lusztig polynomials, of numbers needed to make Conjecture II
explicit. All three Conjectures I,II, III,
as well as Conjectures IIa, IIb, hold for the $p=2$ example studied in \S\S5,6, and collapse to known or recently proved
results in the large prime, $p$-regular weight cases.

 Algebras which are Q-Koszul in our sense are also $T$-Koszul in the sense of Madsen \cite{Madsen}. This
implies that the algebra $\Ext^\bullet_A(T,T)$, where $T$ is a full tilting module for $A_0$ (viewed as an
$A$-module), are again $T$-Koszul.
As formulated in Questions \ref{Madsonques}, we do not know if a similar permanence holds for Q-Koszul or standard Q-Koszul algebras. As the discussion
of this paper shows there are a vast number of important examples of algebras which are standard Q-Koszul.
We expect to
return to Questions \ref{Madsonques} and other issues dealing with the product structure of their Ext-algebras in a later paper.  

Another topic for further research is the speculation, sketched in the final Remarks \ref{finremark}, that the conjectures of \S7 may often have explicit applications to computing Ext-groups between irreducible modules.

 \medskip\medskip
 \begin{center} {\Large\bf Part I: Q-Koszul Algebras}\end{center}
 
 \medskip\medskip

\section{Q-Koszul Algebras}

As above, $A$ denotes a non-negatively graded, finite dimensional algebra.
 Let $\pi:A\twoheadrightarrow A_0\cong A/\sum_{i>0}A_i$ be the quotient homomorphism. If $d$ is an integer,
let $A_{\geq d}=\sum_{i\geq d}A_i$, and define $A_{<d}$, $A_{\leq d}$, $A_{>d}$ analogously. Similar notations
will be used for graded $A$-modules. The algebra $A_0$ may be regarded itself as graded (and concentrated in
grade 0), and every graded
$A$-module $M=\bigoplus_rM_r$ restricts naturally to a graded $A_0$-module, as does each subspace $M_r$, $r\in\mathbb Z$.

Let $M\in\Agrmod$ be concentrated in grades $\geq r$. Then there is a projective $P\in\Agrmod$ which is 
also concentrated in grades $\geq r$ and a surjective graded homomorphism $P\twoheadrightarrow M$. (One can even
assume $P$ is a projective cover of $M$. See \cite[Rem. 8.4]{PS13a} for more discussion.)
Thus, the kernel of the map $P\twoheadrightarrow M$ is a graded module concentrated in grades $\geq r$.  This process can be
continued in the evident way to obtain a graded projective resolution of  $M$ in which each term is
concentrated in grades $\geq r$. A useful consequence is that, if $X$ (respectively, $Y$) is a graded $A$-module concentrated in grades $\geq r$ (respectively, $\leq s$), then, for any non-negative integer $n$,
\begin{equation}\label{earlyremark} \grExt^n_{A}(X,Y)\not =0\implies r\leq s.\end{equation}

 Any $A_0$-module $M$ can be regarded as a graded $A$-module concentrated in grade 0 by making
 $A$ act on $M$ through $\pi$. Thus, there is an exact, additive functor
 \begin{equation}\label{i*}
 i_*:A_0{\text{\rm--mod}}\longrightarrow \Agrmod.\end{equation} 
 Usually, $i_*M$ is denoted simply as $M$ again. Of course, given $A_0$-modules $X,Y$, this
 induces a linear map (still denoted $i_*$)
 \begin{equation}\label{Extmap}i_*: \Ext^r_{A_0}(X,Y)\to\grExt^r_A(X,Y), \quad \forall r \geq 0.\end{equation}
 It is clear that, for $r=0,1$, this map is an isomorphism.   

Every irreducible graded $A$-module 
has the form $L\langle m\rangle$, $m\in\mathbb Z$, where $L$ is an irreducible $A$-module concentrated
in grade 0. In fact, the irreducible $A$-modules concentrated in grade 0 identify with the
irreducible $A_0$-modules.  (However, we do {\it not} assume that $A_0$ is semisimple.) Let $\Lambda$ be a fixed set indexing the distinct isomorphism classes
of irreducible $A_0$-modules.

Suppose that, in addition, $A_0$ is a quasi-hereditary algebra, defined by a poset structure $\leq$ on
$\Lambda$. Thus, $A_0$ has standard (respectively, costandard, irreducible) modules $\Delta^0(\lambda)$ (respectively,
$\nabla_0(\lambda)$, $L(\lambda)$), $\lambda\in\Lambda$, satisfying the usual axioms for a highest weight category; see
\cite{CPS-1}.

The previous paragraph is summarized by saying (as a definition) that $A$ is a {\it $0$-Q-Koszul algebra.} More generally, for $n\geq0$, 
$A$ is an {\it $n$-Q-Koszul algebra} provided that $A_0$ is quasi-hereditary as above, and, for  all $\lambda,\mu\in
\Lambda$,  and all $j\in\mathbb Z$, if $0<i\leq n$, then
\begin{equation}\label{nQ} \forall j\in{\mathbb N},\quad \grExt^i_A(\Delta^0(\lambda),\nabla_0(\mu)\langle j\rangle)\not=0\implies i=j.
\end{equation}
Equivalently, using the isomorphism (\ref{gradedExt}), this means that, for $0\leq i\leq n$, 
\begin{equation}\label{graded2}\Ext^i_A(\Delta^0(\lambda),\nabla_0(\mu))\cong\grExt^i_A(\Delta^0(\lambda),\nabla_0(\mu)\langle i\rangle).
\end{equation}
When $A_0=k$, the notion of an $n$-Q-Koszul algebra identifies with the notion of an $n$-Koszul algebra
defined in \cite[p. 29]{PP}.

A graded algebra $A$ is called Q-Koszul provided that it is $n$-Q-Koszul for all integers
 $n\in{\mathbb N}$.  In other words, condition (\ref{QKoszul}) holds. The notion of a Q-Koszul algebra is left-right symmetric as is the notion of
 standard Q-Koszul introduced in \S3. We generally prefer to work with left modules.
 
   \begin{thm}\label{derived}  (a) Assume that $A$ is $n$-Q-Koszul for some fixed integer $n\geq 1$.  For $A_0$-modules $X,Y$,  the map  (\ref{Extmap}) for $r\leq n$ is an isomorphism
$$ i_*: \Ext^r_{A_0}(X,Y)\overset\sim\longrightarrow\grExt^r_A(X,Y).$$
 (b)  Now assume that $A$ is Q-Koszul. 
 Then the natural functor $i_*:A_0{\text{\rm --mod}}
 \to A{\text{\rm --grmod}}$  in (\ref{i*}) induced by the quotient map $A\to A/A_{\geq 1}\cong A_0$ of graded
 algebras induces a full embedding
 $$i_*:D^b(A_0{\text{\rm --mod}})\to D^b(\Agrmod)$$
 of derived categories. 
 \end{thm}
 
 \begin{proof} Statement (b) follows from a well-known argument, once (a) is established. To prove (a),
 assume that $A$ is $n$-Koszul.  The  map (\ref{Extmap}) is an isomorphism trivially if $n=0,1$  as noted
 after (\ref{Extmap}). So assume $n>1$ and
 proceed by induction on $n$. Let $0\to K\to P\to M\to 0$ be an exact sequence in $A_0$-mod where $P$ is
 $A_0$-projective. Let $I$ be an $A_0$-module having a $\nabla^0$-filtration. Then
 $\grExt^m_A(P,I)=0=\grExt^m_{A_0}(P,I)$ for $m=n-1,n$, using the $n$-Q-Koszul property, since $P$ has a $\Delta^0$-filtration. By the long exact sequence of cohomology,
 there is a commutative diagram 
 \[ \begin{CD} 
 0 @>>> \Ext^{n-1}_{A_0}(K,I) @>\sim>> \Ext^n_{A_0}
 (M,I) @>>>0\\
 @. @VVV  @VVV @.\\ 
 0 @>>> \grExt^{n-1}_A(K,I) @>\sim>>\grExt^n_A(M,I)@>>> 0  
 \end{CD}\]
 in which the two rows are necessarily isomorphisms. By induction, the left hand vertical map is an isomorphism
 so the right hand vertical map  
$ \Ext^n_{A_0}(M,I)\to \grExt^n_A(M,I)$ is an isomorphism.  
 
 This proves
 (a) in case $N=I$ has a $\nabla_0$-filtration. So now assume that $N$ is arbitrary, and form an
 exact sequence $0\to N\to I\to C\to 0$ in $A_0$-mod, where $I$ is $A_0$-injective. Thus, $I$ has
 a $\nabla_0$-filtration, so we again get a commutative diagram
 \[\begin{CD}
  0@>>> \Ext^{n-1}_{A_0}(M,C) @>\sim>> \Ext^n_{A_0}(M,N) @>>>0 \\
 @. @VVV @VVV @.\\
 0 @>>> \grExt^{n-1}_A(M,C) @>\sim>>\grExt^n_A(M,N)@>>> 0 
 \end{CD}\]
 in which the horizontal maps are isomorphisms.  The left hand vertical map is an isomorphism by induction, so the right hand vertical
map is also an isomorphism as required.
  \end{proof}
 
The next two results consider the special cases in which $A$ is 1- or 2-Q-Koszul.  The overall outline and some
of the proof are influenced by the work of Beilinson-Ginzburg-Soergel \cite[\S2.3]{BGS} in the Koszul case, though our situation is more involved. A positively graded algebra
$A$ is, by definition, {\it tight} if it is generated by $A_0$ and $A_1$. Observe this implies (and is equivalent to) the
statement
$A_n=\underbrace{A_1\cdots A_1}_n$, for all $n\geq 1$. Also,  $A_m\cdot A_n=A_{m+n}$ for all $m,n\in\mathbb N$.

\begin{thm} \label{firstprop} Assume that $A$ is 1-Q-Koszul as above, that is, $A_0$ is quasi-hereditary, and for $\lambda,\mu\in\Lambda$, and for all integers $m$, 
$$\grExt^1_A(\Delta^0(\lambda),\nabla_0(\mu)\langle m\rangle)\not=0\implies m=1.$$
Then the graded algebra $A$ is tight.
 \end{thm}

\begin{proof} The exact sequence $0\to A_{\geq 1}\to A\to A_0\to 0$ of graded $A$-modules gives
an exact sequence
\[\begin{aligned} \hom_A(A,\nabla_0(\mu)\langle m\rangle)\overset\alpha\longrightarrow &\hom_A(A_{\geq 1},\nabla_0(\mu)\langle m\rangle)\\
&\overset\beta\longrightarrow\grExt^1_A(A_0,\nabla_0(\mu)\langle m\rangle)\overset\gamma\longrightarrow\grExt^1_A(A,\nabla_0(\mu)\langle m\rangle)=0\end{aligned}\]
for all integers $m\geq 0$. The map $\alpha$ is necessarily 0 for all $m$: consider first the case $m=0$ (where
$\grHom_A(A_{\geq 1}, \nabla_0(\mu))=0$),
and then $m\geq 1$ (where $\grHom_A(A,\nabla_0(\mu)\langle m\rangle)=0$). Hence, $\beta$ is an isomorphism for all $m\geq 0$. Since $A$ is 1-Q-Koszul, it follows that 
$\hom_A(A_{\geq 1},\nabla_0(\mu)\langle m\rangle)=0$ if $m>1$.  (Observe that $A_0$ has a $\Delta^0$-filtration.)  

Let $T$ be the (graded) left ideal of $A$ generated by $A_1$. To show that $A$ is generated by $A_0,A_1$,
it suffices to prove that $T=A_{\geq 1}$. If not, then for some $m>1$ and $\mu$, $\hom_A(A_{\geq 1}/T,
\nabla_0(\mu)\langle m\rangle)\not=0$. Hence, $\hom_A(A_{\geq 1},\nabla_0(\mu)\langle m\rangle)\not=0$,
a contradiction. \end{proof}

Let $A$ be a positively graded algebra and let
$$T_{A_0}(A_1):=\bigoplus_{n\geq 0}\underbrace{A_1\otimes_{A_0}\cdots\otimes_{A_0}A_1}_n$$
be the tensor algebra of the $(A_0,A_0)$-bimodule $A_1$ (with the term for $n=0$ set to be $A_0$). Generalizing  the usual definition, the graded algebra $A$ is defined to be quadratic if the multiplication map $m:T_{A_0}(A_1)\to A$, defined by $a_1\otimes\cdots\otimes a_n\mapsto a_1\cdots a_n$,
is surjective, and if the relation ideal $I:=\ker\, m$ is generated by its grade 2 component. That
is, $I$ is generated by  the kernel $W_2$ of the multiplication map $A_1\otimes_{A_0}A_1\to A_2$.  Since $A_1$ is an $(A_0,A_0)$-bimoudle, it is a left module for both algebras
$A_0$ and $A_0^{{\text{\rm op}}}$ (and, of course, the two actions commute). For $\lambda\in\Lambda$, the
corresponding standard module $\Delta^{0,\text{\rm op}}(\mu)$ for $A^{0,\text{\rm op}}$ is defined to be a linear dual $\nabla(\lambda)^*$, 
viewed as a left $A_0^{\text{\rm op}}$-module. It has irreducible head $L^{\text{\rm op}}(\lambda)=L(\lambda)^*$.
\begin{thm} \label{thirdprop} Assume that $A$ is 2-Q-Koszul. (Thus, $A$ is also 1-Q-Koszul.)  
Then the following statements hold.

(a) $A$ is a quadratic algebra.

(b) The subsapce
$W_2$ of $A_1\otimes_{A_0}A_1$  defined above generates the kernel of the multiplication map
$A\otimes_{A_0}A_1\to A_{\geq 1}$ (respectively, $A_1\otimes_{A_0}A\to A_{\geq 1}$) as a left $A$-module (respectively,
as a left $A^{\text{\rm op}}$-module). 

(c)  The left $A_0$-module $A_1$ has a $\Delta^0$-filtration. Also, the left $A_0^\op$-module $A_1$  has
a $\Delta^{0,\op}$-filtration. 
  \end{thm}

\begin{proof} We first prove (c). The long exact sequence of $\grExt_A^\bullet(-,\nabla_0(\mu)\langle r\rangle)$ for the exact sequence
$0\to A_{\geq 1}\to A\to A_0\to 0$ in $\Agrmod$  gives
\begin{equation}\label{eq1} \grExt^1_A(A_{\geq 1},\nabla_0(\mu)\langle1\rangle)\cong\grExt^2_A(A_0,\nabla_0(\mu)\langle 1\rangle)=0.\end{equation}
(The term on the right is 0 since, by hypothesis, $A$ is 2-Q-Koszul.) 

Next, again using the long exact sequence of
$\grExt^\bullet_A(-,\nabla_0(\mu)\langle r\rangle)$ for the exact sequence $0\to A_{>1}\to A_{\geq 1} \to A_1\to 0$ gives
an exact sequence
$$\begin{aligned}
 \grHom_A(A_{>1},\nabla_0(\mu)\langle1\rangle) &\to\grExt^1_A(A_1,\nabla_0(\mu)\langle 1\rangle)\\
&\to \grExt^1_A(A_{\geq 1},\nabla_0(\mu)\langle 1\rangle)\to\grExt^1_A(A_{>1},\nabla_0(\mu)\langle 1\rangle).
\end{aligned}$$
Obviously, the left hand end of the above exact sequence vanishes, while, from (\ref{earlyremark}), the right hand end 
is also
0. Thus, using (\ref{eq1}), 
\[0=\grExt^1_A(A_{\geq 1},\nabla_0(\mu)\langle 1\rangle)
\cong\grExt^1_{A}(A_1,\nabla_0(\mu)\langle 1\rangle).\]
Theorem \ref{derived}(a), applied to $\grExt^1_A(A_1\langle -1\rangle,\nabla_0(\mu))=\grExt^1_A(A_1,\nabla_0(\mu)\langle 1\rangle)=0$, gives (using (\ref{gradedExt}))
\[\Ext^1_{A_0}(A_1,\nabla_0(\mu))=0,\quad\forall\mu\in\Lambda.\]
 This means that $A_1$ has
a $\Delta^0$-filtration, as required by (c).

Next, we prove (a) and (b). By Theorem \ref{firstprop}, there is an exact sequence
\begin{equation}\label{seq}
0\to W\to A\otimes_{A_0}A_1\overset\phi\to A_{\geq 1}\to 0\end{equation}
of graded left $A$-modules.
Since $\phi$ is an isomorphism in grade 1, the graded $A$-module $W$ is concentrated in
grades $\geq 2$. The following claim is needed for the proof of (a); applied together with its analog for $A^{\text{\rm op}}$, it also gives (b) immediately.

\medskip\noindent
\underline{Claim:} $W$ is generated in grade 2 as a left $A$-module, i.~e., $W=AW_2$.
\medskip

Before proving the Claim, we show that it implies (a), that is, that $A$ is quadratic. By construction, $W_2\subseteq A_1\otimes A_1\subseteq T:=T_{A_0}(A_1)=\bigoplus_{s\geq 0}A_1^{\otimes s}$, the tensor algebra of $A_1$ over $A_0$. (Thus, 
$\otimes:=\otimes_{A_0}$ in this proof.)
Here $A_1^{\otimes 0}:=A_0$. Let $I$ be the kernel of the evident algebra surjection $\alpha=\bigoplus_s\alpha_s:T\twoheadrightarrow A$.
It suffices to prove that $I=\langle W_2\rangle$ (the ideal in $T$ generated by $W_2\subseteq T$). Obviously, $I_0=0=I_1$ and $I_2=W_2$. Let $s>2$ be an
integer, and assume by induction that $I_{s-1}=\langle W_2\rangle_{s-1}$. Thus,  
\begin{equation}\label{cdiagram0}\langle W_2\rangle_{s-1}A_1=I_{s-1} A_1=\ker\left(A_1^{\otimes{s-1}}\otimes A_1=T_{s-1}\otimes A_1\overset{\alpha_{s-1}\otimes A_1}\longrightarrow A_{s-1}\otimes A_1\right),\end{equation}
where the two left hand products are taken in $T$.\footnote{Of course, $I_{s-1}\subseteq A_1^{\otimes(s-1)}$, though
there may not be an inclusion of $I_{s-1}\otimes A_1$ into $A_1^{(s-1)}\otimes A_1$ because the functor
 $-\otimes A_1$ might possibly be only right exact. Nevertheless, it makes sense to form the product $I_{s-1}A_1$ in $T$. This
product is the same as the image of the map $I_{s-1}\otimes A_1\to A_1^{\otimes(s-1)}\otimes A_1=T_s$. The right
exactness of $-\otimes A_1$ then gives (\ref{cdiagram0}).}
 We need to show that $I_s=\langle W_2\rangle_s$, or, equivalently, $I_s\subseteq \langle W_2\rangle_s$.  Consider the commutative diagram
\begin{equation}\label{cdiagram}\begin{CD} T_s @<<\sim< T_{s-1}\otimes A_1 @<<\sim< T_{s-2}\otimes A_1\otimes A_1\\
@VV{\alpha_s}V @VV{\alpha_{s-1}\otimes A_1}V  @VV{\beta_s}V \\
A_s@<<<  A_{s-1}\otimes A_1 @=  A_{s-2}(A_1\otimes A_1)
\end{CD}
\end{equation}
where each map in the top row is induced by multiplication in $T$.\footnote{In some sense, these multiplication maps
are just equalities, but it is useful in the proof to keep them separate.}  Also, $\beta_s$ is the composite of the
surjective map
$\alpha_{s-2}\otimes A_1\otimes A_1$ with module multiplication of $A_{s-2}$ on $A_1\otimes A_1\subseteq
A\otimes A_1$. 
Let
$x\in I_s\subseteq T_s$, and let $x'$ be the corresponding element in $A_1^{\otimes (s-1)}\otimes A_1=T_{s-1}\otimes A_1$.  Then $(\alpha_{s-1}\otimes A_1)(x') $
maps to 0 under the multiplication map $A_{s-1}\otimes A_1\to A$, since $x\in I_s$. Thus, $(\alpha_{s-1}\otimes
A_1)(x')\in W_s$, which equals $A_{s-2}W_2$ by the Claim. Since $\alpha_{s-2}$ is surjective, there is
an element $y^\sharp\in A_1^{\otimes (s-2)}W_2=T_{s-2}W_2\subseteq T_{s-2}\otimes A_1\otimes A_1$
with image $(\alpha_{s-1}\otimes A_1)(x')\in W_s=A_{s-2}W_2$ under the map $\beta_s$. Let $y'\in T_{s-1}\otimes A_1$ and $ y\in T_s$ correspond to
$y^\sharp$ in the top row of (\ref{cdiagram}). The commutativity of (\ref{cdiagram}) 
gives 
$(\alpha_{s-1}\otimes A_1)(y')=(\alpha_{s-1}\otimes A_1)(x')$.
Then $x'-y'\in \ker(\alpha_{s-1}\otimes A_1)$, which by induction is $\langle W_2\rangle_{s-1}\ A_1$, so $x-y\in
\langle W_2\rangle_{s-1}A_1\subseteq T_s$ in (\ref{cdiagram0}).  By construction, $y'$ and $y$ belong
to $\langle W_2\rangle$. So $x\in\langle W_2\rangle$. This completes the inductive step $I_s=\langle W_2\rangle_s$.
Thus, $I=\langle W_2\rangle$, and $A$ is quadratic.

It remains to check the Claim (which will also prove (b)). First, there is an injection 
\[\label{injection}\hom_A(W, \nabla_0(\lambda)\langle r\rangle)\hookrightarrow\grExt_A^2(A_0,\nabla_0(\lambda)\langle r\rangle),\quad\forall r\in\mathbb N.\]
To see this, let $P$ be an $A_0$-projective cover of $A_1$, viewed as graded and concentrated in grade 1.
Then $A\otimes P=A\otimes_{A_0}P$ is a graded projective $A$-module, equipped with a map 
$A\otimes P\twoheadrightarrow A_{\geq1}$ (surjective, by Theorem \ref{firstprop}) sending $1\otimes P$ to $A_1$. Let $\widehat W$ be the kernel of this map. There is a commutative
diagram
\begin{equation}\label{contain} \begin{CD} 
0 @>>>\widehat W@>>> A\otimes P @>>> A  @>>> A_0 @>>> 0\\
 @.
@VVV @VVV @VVV @| @.\\
0 @>>> W @>>> A\otimes A_1 @>>> A @>>> A_0 @>>> 0,\end{CD}
\end{equation}
in which each row is exact and the vertical maps are the evident surjections.  The top row can be used to 
compute $\grExt^2_A(A_0,\nabla_0(\lambda)\langle r\rangle )$ in the usual cocycles/coboundaries
way. 
The 2-cocycles are elements of $\grHom_A(\widehat W,\nabla_0(\lambda)\langle r\rangle)$, and
the image in it of $\grHom_A(A\otimes P,\nabla_0(\lambda)\langle r\rangle)$ is the space of 2-coboundaries.
All elements of $\hom_A(A\otimes P,\nabla_0(\lambda)\langle r\rangle)$ are zero,
unless $r=1$, since $P$ is concentrated in grade 1.   But $\hom_A(W,\nabla_0(\lambda)\langle r\rangle)
=0$ when $r=1$ (since $W_m=0$ for $m\leq 1$). The composite of the maps
\begin{equation}\label{composite} \hom_A(W,\nabla_0(\lambda)\langle r\rangle)\to\hom_A(\widehat W,\nabla_0(\lambda)\langle r\rangle)
\to\grExt^2_A(A_0,\nabla_0(\lambda)\langle r\rangle)\end{equation}
is an injection in all cases. 

Observe that $W/AW_2$ is a (positively) graded $A$-modules, vanishing
in grades $\leq 2$. If $W/AW_2\not=0$, we have $(W/AW_2)_s\not=0$ for some minimal integer $s$.
Necessarily $s>2$. Also, 
$(W/AW_2)_s=(AW_2+W_{\geq s})/(AW_2+W_{>s})$ is an $A$-module killed by $A_1$. Choose
any irreducible graded $A_0=A/A_{\geq 1}$-module $L(\lambda)\langle s\rangle$  in the head of $(W/AW_2)_s$. Then $\hom_A(W,\nabla_0(\lambda)\langle s\rangle)\not=0$, so
 $\grExt^2_A(A_0,\nabla_0(\lambda)\langle s\rangle)\not=0$ by (\ref{composite}). Since $s>2$, this contradicts the hypothesis is 2-Q-Koszul
. This contradiction shows that $W=AW_2$ and completes the proof of the Claim. Thus, (a) and (b) are
 now proved, as well as the theorem.
  \end{proof}
 
If $V$ is an $A_0$-module, let $\cdots\to P^1\to P^0\twoheadrightarrow V$ be a minimal projective resolution
in the category of $A_0$-modules. Recall, for any $m>0$,  the $m$th syzygy module $\Omega_m(V)$, for the
algebra $A_0$, is the
kernel of the map $P^{m-1}\to P^{m-2}$ (setting $P^{-1}:=V$). By convention, $\Omega_0(V)=V$.
 
\begin{lem}\label{biglemma}
Assume that $A$ is $(n+1)$-Q-Koszul for some integer $n>0$. Let $s,r$ be non-negative integers with $s\leq r$
and $0<s\leq n$.

(a) Then $\Ext^r_{A_0}(A_s,\nabla_0(\gamma))=0$
for all $\gamma\in\Lambda$. In particular, 
 the
$A_0$-modules
$$\Omega_{r-1}(A_s), \Omega_r(A_s), \Omega_{r+1}(A_s), \cdots $$
all have $\Delta^0$-filtrations. Consequently (using the case $r=s$), the $A_0$-modules
$$\Omega_{s-1}(A_s), \Omega_s(A_s), \Omega_{s+1}(A_s), \cdots$$
all have $\Delta^0$-filtrations.
 
(b) Assume also $r\leq n+1$, and let $m$ be any integer such $m\leq r$. Then 
\begin{equation}\label{vanishing2} \grExt_{A}^{r}(A_{s},\nabla_{0}(\nu)\langle m\rangle)=0,\quad\forall\nu\in\Lambda.\end{equation} 

\end{lem}

\begin{proof} Before beginning the proof, we observe an additional consequence (c) of the hypotheses of the
lemma. It is obvious, since
$A_0$ has a $\Delta^0$-filtration (as a left $A_0$-module) by standard properties of QHAs.

\smallskip
(c) Let $r$ be a nonnegative integer $\leq n+1$. Then (\ref{vanishing2}) holds for
$s=0$ and all $m\not=r$. 
\smallskip

 In order to prove (a) for a given $s>0$, it is sufficient to prove the shorter statement:

\smallskip\noindent $(a')\quad \Ext^s_{A_0}(A_s,\nabla_0(\nu))=0,\quad \forall \nu\in\Lambda.$

\smallskip
 In fact, assume that (a$^\prime$) holds for $s$. Dimension shifting gives
$\Ext_{A_0}^1(\Omega _{s-1}(A_s),\nabla_0(\nu))=0$, for all $\nu\in\Lambda$. Thus,  $\Omega_{s-1}(A_{s})$ has a $\Delta^{0}$
-filtration. This also implies the higher syzygies $\Omega_{s}(A_{s}),$ $\Omega_{s+1}(A_s),$, \dots
have $\Delta^{0}$-filtrations, or, equivalently, 
$\Ext_{A_{0}}^{r}(A_s,\nabla_{0}(\nu))=0$ for all $\nu\in\Lambda$ and $r\geq s$. This
completes the proof that  $(a') \implies (a)$, for any given $s>0$. The opposite implication, which is not used below, is obvious.

We prove part (a) by induction on $n$. The main isomorphism we develop will also help to prove part (b). Thus, if $0< s<n$, 
(a) holds as written for $s$ by induction. In particular, 
 the $A_0$-syzygy $\Omega_{s-1}(A_{s})$ 
has a $\Delta^{0}$-filtration. Once part (a) has
been proved for $n$, we can also allow $s=n$ in this statement, i.~e., we will be able to conclude
that $\Omega_{n-1}(A_n)$ has a
$\Delta^0$-filtration. 

Using the map $\pi:A\to A_0$, regard  $X:=\Omega_{s-1}(A_s)$ as an object in $\Agrmod$ which is pure of grade 0. 
Thus, $X$ is the end term in an evident partial resolution of 
$A_{s}\langle-s\rangle$, the latter viewed as a purely graded $A$-module 
$(A_{\geq s}/A_{>s})$ $\langle-s\rangle$ of grade $0$. The intermediate terms
in this partial resolution are projective $A_{0}$-modules $P$, viewed as
purely graded $A$-modules of grade $0$. With the latter interpretation of
$P$, we have, for any integer $m$, $\grExt_{A}^{r}(P,\nabla_{0}(\nu)\langle m-s\rangle)=0$
whenever $n+1\geq r>m-s$, since $P$ has a $\Delta^{0}$-filtration, and we
have assumed the $(n+1)$-Q-Koszul property. This vanishing may be used to iteratively
dimension shift, starting with $r\leq n+1$ and ending with $r-(s-1)$,
provided $r-(s-1)>m-s$ and $r-(s-1)>0$. Equivalently, $r\geq m$ and $r\geq s$.  This gives the lower
isomorphism below, for $m\leq r\leq n+1$, $s\leq r$, and
$0<s<n$,  
\begin{equation}\label{displayed}\begin{aligned}
\grExt_{A}^{r}(A_{s},\nabla_{0}(\nu)\langle m\rangle ) &\cong
\grExt_{A}^{r}(A_{s}\langle-s\rangle,\nabla_{0}(\nu)\langle
m-s\rangle) \\ &\cong\grExt_{A}^{r-s+1}(X,\nabla_{0}(\nu)\langle
m-s\rangle).\end{aligned}
\end{equation}
The lower term is $0$ in all these
cases, since $X$ has a $\Delta^{0}$-filtration, and $r-s+1\neq m-s$.  In particular, these isomorphisms and vanishings hold for $r=n+1$ and $m=n$, for any 
positive integer $s<n$. Also,
with the (same) values $r=n+1$ and $m=n$, we have the additional vanishing
$\grExt_{A}^{r}(A_{s},\nabla_{0}(\nu)\langle m\rangle)= 0$ when $s=0$, by (c) above. Consequently, noting that the graded quotient $A_{<n}$ of the graded
$A
$-module $A$ is filtered by the graded $A$-modules $A_{s}$, $0\leq s<n$, we
have (as in the proof of Theorem \ref{thirdprop}(c))
\begin{equation*}
\grExt_{A}^{n}(A_{n},\nabla_{0}(\nu)\langle n\rangle )\cong
\text{\textrm{ext}}_{A}^{n}(A_{\geq n},\nabla_{0}(\nu)\langle
n\rangle)\cong
\grExt_{A}^{n+1}(A_{<n},\nabla_{0}(\nu)\langle n\rangle)=0.
\end{equation*}
 Hence, ${\text{\textrm{Ext}}}_{A_{0}}^{n}(A_{n},\nabla_{0}(\nu))=0$ by
Theorem \ref{derived}. This proves (a$^\prime$) and, thus, statement (a) for $s=n$, completing the inductive
step. This proves (a). 

In particular, we can now use the case $s=n$ in the above
discussions.  The displayed isomorphisms (\ref{displayed}), now allowing $s=n$ as well, give all the vanishings required by part (b).
This completes the proof.
\end{proof}

Let $A$ is a QHA with weight poset $\Lambda$, and let $\Gamma$ be a non-empty ideal
in $\Lambda$.  For any nonempty poset ideal $\Gamma$  in $\Lambda$, let $A_\Gamma$ be the largest quotient
algebra of $A$ whose modules consist of all finite dimensional $A$-modules with composition factors $L(\gamma)$,
$\gamma\in\Gamma$. 

\begin{lem}\label{ideallemma} Let $m$ be a non-negative integer. Assume $A$ is an algebra that is both $m$-Q-Koszul and quasi-hereditary with weight poset $\Lambda$.
Let $\Gamma$ be a non-empty poset ideal in $\Lambda$. Then $A_\Gamma$ is also $m$-Q-Koszul and quasi-hereditary, with weight poset $\Gamma$.\end{lem}
 
 \begin{proof} First, $A_\Gamma=A/J$ for some idempotent ideal $J$ in $A$. Then \cite{CPS1a} implies that
 $J=AeA$ for some $e\in A_0$. It follows that $A_\Gamma$ is positively graded, and the natural map
 $A\to A_\Gamma$ is a homomorphisms of graded algebras.  If $\gamma\in\Gamma$, $\Delta^0(\gamma)$
 is the standard object corresponding
  for $(A_\Gamma)_0$ to the weight $\gamma$. Similarly, $\nabla_0(\gamma)$ is the costandard object
  for $(A_\Gamma)_0$. Now the result follows from the naturally of (\ref{gradedExt}), together with standard recollement properties of QHAs.\end{proof}
  
  \begin{rem} Although it is not used in this paper, it can be easily shown that, in the notation of the above
  proof, the algebra $eAe$ is $m$-Q-Koszul and quasi-hereditary. The module category $eAe$-mod
  is equivalent to the quotient category of $\Amod$ by the subcategory which is strict image of $A/J$-mod in
  $\Amod$.  For a similar result, in the setting of 
  standard Q-Koszul algebras, see the ``recollement" discussion at the end of \S3.   \end{rem}
  
  If $A$ is a positively graded QHA with weight poset $\Lambda$, then each standard module $\Delta(\gamma)$
  can be graded $\Delta(\gamma)=\bigoplus_{n\geq 0}\Delta(\gamma)_n$, with $\Delta(\gamma)_0\cong
  \Delta^0(\gamma)$, the standard object for the QHA $A_0$. We have the following result.
 
 \begin{cor}\label{sec2cor} Let $A$ be 2-Q-Koszul and quasi-hereditary with weight poset $\Lambda$.
 
(a) 
 For $\gamma\in\Lambda$, $\Delta(\gamma)_1$ has a $\Delta^0$-filtration.
 
 (b) If $A$ is 3-Q-Koszul, then $\Ext^r_{A_0}(\Delta(\gamma)_2,\nabla_0(\nu))=0$ for all $r\geq 2$ and all
 $\nu\in\Lambda$. 
 
  More generally, assume that $A$ is $(n+1)$-Q-Koszul for some integer $n>0$, in addition to being a QHA, and let
  $s\leq r$ be nonnegative integers with $0<s\leq n$.  Then $\Ext^r_{A_0}(\Delta(\gamma)_s,\nabla_0(\nu))=0$
 for all $\gamma,\nu\in\Lambda$.
 
 In addition, the syzygy modules
 $\Omega_{s-1}(\Delta(\gamma)_s)$, $\Omega_s(\Delta(\gamma)_s),$ $\Omega_{s+1}(\Delta(\gamma)_s)$, \dots
 all have $\Delta^0$-filtrations.
 \end{cor}
 
 \begin{proof}  We first prove (a). If $\gamma\in\Lambda$ is maximal, then $\Delta(\gamma)$ is a projective
 graded $A$-module (with $\Delta(\gamma)_0$ identifying with the $A_0$-head of $\Delta(\gamma)$). Then,
 by Theorem \ref{thirdprop}(c), it follows that $\Delta(\gamma)_1$ has a $\Delta^0$-filtration. If $\gamma$
 is not maximal, we can choose an ideal $\Gamma$ which contains $\gamma$ as a maximal element. Part
 (a) reduces to the case in which $A$ is replaced by $A_\Gamma$, using Lemma \ref{ideallemma}, and (a)
 follows as above.
 
 Part (b) and the remaining paragraph are proved similarly.   \end{proof}

 We can now prove the following result. We assume that $A$ is a positively graded
 QHA which 2-Q-Koszul. As in Theorem \ref{thirdprop}, let $W_2$ be the kernel of the multiplication map $A_1\otimes_{A_0}A_1
 \to A_2$. An $A_0\otimes_kA_0^{\text{\rm op}}$-module (equivalently, and $(A_0,A_0^{\text{\rm op}}$)-bimodule) $M$ has, by definition, a $\Delta^0\otimes_k\Delta^{0,\text{\rm op}}$-filtration
if and only if has a submodule filtration with sections $\Delta^0(\lambda)\otimes_k\Delta^{0,\text{\rm op}}(\mu)$,
for $\lambda,\mu\in\Lambda$. For example, the algebra $A_0$, viewed as an $A_0\otimes_kA_0^{\text{\rm op}}$-module has
a filtration with sections $\Delta^0(\lambda)\otimes_k\Delta^{0,\text{\rm op}}(\lambda)$, $\lambda\in\Lambda$. It will
be useful to keep in mind that the tensor product $A\otimes_k B$ of QHAs $A$ and $B$ over $k$ is again
quasi-hereditary. If $\Lambda_A$ and $\Lambda_B$ are the posets of $A$ and $B$, then $\Lambda=\Lambda_A\times
\Lambda_B$ is the poset of $A\otimes_kB$, with $(\lambda,\lambda')\leq(\mu,\mu')$ if and only if $\lambda\leq \mu$
and $\lambda'\leq\mu'$. The standard (respectively, costandard) modules for $A\otimes_k B$ are tensor products of standard
(respectively, costandard) modules of $A$ with those of $B$. For more details, see \cite{Wied}.

\begin{thm}\label{tensortensor} Assume that $A$ is 2-Q-Koszul and that $A$ is a QHA. 

(a) Then the $(A_0,A_0^{\text{\rm op}})$-bimodule
$A_1$ has a $\Delta^0\otimes_k\Delta^{0,\text{\rm op}}$-filtration. Also, $A_1\otimes_{A_0}A_1$ has a
$\Delta^0\otimes_k\Delta^{0,\text{\rm op}}$-filtration.

(b) Now assume, in addition, that $A$ is 3-Q-Koszul.  Then 
$W_2$ (defined above) has a $\Delta^0\otimes_k \Delta^{0,\text{\rm op}}$-filtration.\end{thm}

\begin{proof}The $A\otimes_k A^\op$-module $A$ has a filtration with sections $\Delta(\lambda)\otimes_k\Delta^\op(\lambda)$, $\lambda\in\Lambda$, as briefly discussed in the proof of Lemma \ref{ideallemma}. Consequently, if $s\geq 0$, the $(A_0\otimes_kA_0^\op)$-module $A_s$ has a filtration with sections $(\Delta(\lambda)\otimes_k\Delta^\op(\lambda))_s$.
 The tensor product  $\Delta(\lambda)\otimes_k\Delta^\op(\lambda)$ arises as part of the image of a product
 $Ae\cdot eA$ in a quotient $A/J$ of $A$ by a graded idempotent ideal $J$. In fact, $\Delta(\lambda)$ arises as
 the image of $Ae$, and $\Delta^\op(\lambda)$ identifies with the image of $eA$. Clearly, $(Ae)_i\cdot (eA)_j
 \subseteq (AeA)_{i+j}$. Consequently, there is an identification of  $A_0\otimes_kA_0^\op$-modules 
 \begin{equation}\label{bigdirect} (\Delta(\lambda)\otimes_k\Delta^\op(\lambda))_s\cong \bigoplus_{i+j=s}\Delta(\lambda)_i\otimes_k\Delta^\op(\lambda)_j.\end{equation}
 
 For $s=1$, it is now clear that $A_1$ has a $\Delta^0\otimes_k\Delta_0^\op$-filtration. The proves the first
 statement in (a). The second assertion in follows from
 the fact that
 $$\Delta^0(\lambda)\otimes^{\mathbb L}_{A_0}\Delta^{0,\op}(\mu)\cong\begin{cases} k,\quad {\text{\rm if $\lambda=\mu$}};\\
 0\quad{\text{\rm otherwise.}}\end{cases}.$$
 See \cite[Prop. 9.1]{PS13a}.

 To begin the proof of (b), we first apply the K\"unneth
 formula to the terms in the direct sum on the right hand side of (\ref{bigdirect}), to obtain
\begin{equation}\label{Kunneth}\begin{aligned}\Ext^m_{A_0 \otimes_kA_0^\op}&(\Delta(\lambda)_i\otimes_j\Delta^\op(\lambda)_j, \nabla_0(\nu)\otimes_k\nabla_{0,\op}(\mu))\cong \\
& \bigoplus_{u+v=m} \Ext^u_{A_0}(\Delta(\lambda)_i,\nabla_0(\nu))\otimes_k\Ext^v_{A^\op_0}(\Delta^\op(\lambda)_j,\nabla_{0,\op}(\mu)).\end{aligned}\end{equation}

We know that $\Ext^u_{A_0}(\Delta(\lambda)_i,\nabla_0(\nu))=0$ for $0<i\leq u$ and $i\leq 2$,  by
Corollary \ref{sec2cor}. A similar vanishing holds, of course, for $i=0$, if $u>0$. Also, we can work with
$A_0^\op$, to obtain that $\Ext^v_{A^\op_0}(\Delta^\op(\lambda)_j,\nabla_{0,\op}(\nu))=0$ for $0<j\leq v$ and $j\leq 2$.

\smallskip
\noindent
\underline{Claim:} $\Ext^2_{A_0\otimes_kA_0^\op}(A_2,\nabla(\lambda)\otimes_k\nabla_{0,\op}(\nu))=0,\quad \forall\lambda,\nu.$

\smallskip
To prove this, we take $m=s=2$ in  (\ref{bigdirect}) and (\ref{Kunneth}). Because, as noted above, $A_2$ has a filtration with
sections $(\Delta(\lambda)\otimes_k\Delta^\op(\lambda))_s$, it suffices to show each term in the sum 
(\ref{Kunneth})  is 0
when $i+j=s$, $i,j\geq 0$. This is
clear if either $u=2$ or $v=2$ from the vanishing results immediately above the Claim. The other case is
$u=v=1$. Then, if $i=j=1$, we are done by Corollary \ref{sec2cor}. Otherwise, either $i=0$ or $j=0$, and
the discussion immediately above the Claim again applies.

Finally, to complete the proof, consider the short exact sequence $0\to W_2\to A_1\otimes_{A_0}A_1\to
A_2\to 0$ and the resulting long exact sequence of $\Ext^\bullet_{A_0\otimes A_0^\op}(-, \nabla_0(\lambda)\otimes_k\nabla_{0,\op}(\mu)$. 
Now (b) follows from the Claim and part (a).
\end{proof}
  
  \begin{rem} In general, it may not true that $\Delta(\lambda)_2$ has a $\Delta^0$-filtration under the hypothesis
  of Theorem \ref{tensortensor}(b). This follows
  from the discussion of Turner's counterexample in section \ref{turner} below. In addition, this means that $A_2$ may
  not have a $\Delta^0$-filtration under these hypotheses. For graded algebras $A$ arising from semisimple
  algebraic groups and
  finite posets of $p$-regular weights,
  all the $A_0$-modules $A_r$ have $\Delta^0$-filtrations if $p\gg 0$, using \cite[Thm. 5.1]{PS11}. 
  (Using the argument from the proof of  Theorem \ref{tensortensor}, it follows each $A_r$ has, in fact, a $\Delta^0\otimes_k \Delta^{0,\op}$-filtration, under
  this $p\gg 0$, $p$-regular weight hypothesis.)
  However, our aim in  this section has been to develop a theory which might hold for small primes, including
  even $p=2$ in type $A$, where Turner's counterexample occurs. See the conjectures in Section 7. The broad
  class of examples proposed there is expected to be at least Q-Koszul and quasi-hereditary, and even satisfy
  the stronger ``standard Q-Koszul" property discussed in the next section.
     \end{rem}
  
  \smallskip
  \begin{prob}
Let $M$ be a graded module for a Q-Koszul algebra $A$. Give conditions on a resolution of $M$ equivalent to
the condition that $\grExt^m(M,\nabla_0(\lambda)\langle r\rangle)\not=0\implies m=r$, for 
any $\lambda\in\Lambda$. \end{prob}

\section{Standard Q-Koszul algebras} In this section, standard Q-Koszul algebras are defined. The
definition simplifies that given in \cite{PS13a}, but a main result establishes the two different
notions are the same.

Suppose that $B$ is a QHA with weight poset $\Lambda$ over an algebraically closed field. 
Let $\sC=\sC(B)$ be the (highest weight) category of finite dimensional $B$-modules.  If $\Gamma$ is a
non-empty ideal in $\Lambda$, let $\sC[\Gamma]$ be the full subcategory of $B$-modules which have
composition factors $L(\gamma)$, $\gamma\in\Gamma$. Of course, $\sC[\Gamma]=\sC(B/J)$, for
a suitable defining ideal $J=J(\Gamma)$ of $B$. Necessarily $B/J$ is a QHA and $\sC[\Gamma]$ is
a highest weight category with weight poset $\Gamma$.  For details, see \cite{CPS-1}.

If $B=\bigoplus_{n\geq 0}B_n$ is positively graded, let $\sC_\gr=\sC_\gr(B)$ be the category of finite dimensional $\mathbb Z$-graded $B$-modules. (Sometimes, we also denote $\sC_\gr$ by $B$-grmod.) By \cite[Prop. 4.2]{CPS1a}, 
the idempotent ideal $J=J(\Gamma)$ is homogeneous; in fact, $J=BeB$ for some idempotent $e\in B_0$.
Each standard module
$\Delta(\lambda)$, $\lambda\in\Lambda$, has a natural $\mathbb N$-grading $\Delta(\lambda)=\bigoplus_{i\geq 0}\Delta^i(\lambda)$ in which each $\Delta_i(\lambda)$ is naturally a $B_0$-module. 
 Similarly, the costandard module $\nabla(\lambda)$ has a grading $\nabla(\lambda)=\bigoplus _{i\leq 0}
\nabla_i(\lambda)$.

A proof of the following elementary result is found in \cite[Cor. 3.2]{PS11}.

\begin{lem}
Suppose $B=\bigoplus_{n\geq 0}B_n$ is a positively graded quasi-hereditary algebra with poset $\Lambda$. Then
the subalgebra $B_0$ is quasi-hereditary with poset $\Lambda$. For $\lambda\in\Lambda$, the corresponding
standard (respectively, costandard) module is $\Delta^0(\lambda)$ (respectively, $\nabla_0(\lambda))$.\end{lem}

\begin{defn} The graded quasi-hereditary algebra $B$ is called a {\it standard Q-Koszul algebra} provided
that, for all $\lambda,\mu\in\Lambda$, 
\begin{equation} \label{conditions} \begin{cases}(a)\quad\grExt^n_B(\Delta(\lambda),\nabla_0(\mu)\langle r\rangle)\not=0
\implies n=r;\\
(b)\quad \grExt^n_B(\Delta^0(\mu),\nabla(\lambda)\langle r\rangle)\not=0\implies n=r\end{cases}\end{equation}
for all integers $n,r$.\end{defn}

Consider the bounded derived category $D^b(\sC_\gr)$ for the abelian category $\sC_\gr=\sC_\gr(B)$ of 
(finite dimensional) graded $B$-modules. Then $D^b(\Bgrmod)$ is a triangulated category with
shift operator $X\mapsto X[1]$, $X\in D^b(\Bgrmod)$. If $X=X^\bullet$ is represented by a complex in $D^b(\Bgrmod)$, 
$X[1]\in D^b(\Bgrmod)$ is the complex obtained by shifting $X$ one-unit to the left and replacing each
differential by its negative. Put $[r]=[1]^r$. Next, set $X\langle r\rangle$ to be the complex by applying
the grading shift operator $\langle r\rangle$ to the terms $X^n$ and the differentials. Finally, let
$X\{r\}:=X\langle r\rangle[r]=X[r]\langle r\rangle$.

Define a full subcategory $\sE^L=\sE^L(\sC_\gr):=\bigcup_{i\geq 0} \sE^L_i$ of $D^b(\sC_\gr)$ as follows: Let $\sE^L_0\subseteq D^b(\sC_\gr)$ consist of all finite direct sums 
$\Delta(\lambda)\{r\}$, for $\lambda\in\Lambda$, $r\in\mathbb Z$. Having defined $\sE^L_r$ define
$\sE^L_{r+1}$ to consist of all objects $X\in D^b(\sC_\gr)$ for which there is a distinguished
triangle $Y\to X\to Z\to$ with $Y,Z\in \sE^L_r$.  Another full subcategory $\sE^R:=\sE^R(\sC_\gr)=\bigcup_{i\geq 0}\sE^R$ of
$D^b(\Bgrmod)$ is constructed similarly, but using the $\nabla(\lambda)\{r\}$, $\lambda\in\Lambda$,
$r\in\mathbb Z$.

For $X,Y\in D^b(\Bgrmod)$ and $n\in\mathbb Z$, write $\grHom^n(X,Y):=\Hom_{D^b(\sC_\gr)}(X,Y[n])$.
If $X,Y\in\sC_\gr$ are viewed as complexes concentrated in grade 0, then $\hom^n(X,Y)=\grExt^n_B(X,Y)$
for $n\geq 0$, and $=0$ if $n<0$.


\begin{thm} \label{better} Given $M\in D^b(\sC_\gr)$, 
\[\begin{cases} (a)\quad M\in\sE^L\iff \forall \lambda\in\Lambda, n,r\in{\mathbb Z},
\grHom^n_{D^b(\sC_\gr)}(M,\nabla(\lambda)\langle r\rangle)\not=0\implies n=r\\
(b)\quad M\in\sE^R\iff \forall \lambda\in\Lambda, n,r\in{\mathbb Z}, \grHom^n_{D^b(\sC_\gr)}(\Delta(\lambda)\langle r\rangle,M)\not=0\implies n=-r.\end{cases}\]
 \end{thm}
 
 \begin{proof} We will prove statement (a), leaving the similar (b) to the reader. For $M\in D^b(\sC_\gr)$,
 we say that condition $(\bigstar(M))$ holds provided: 
 $$\forall \lambda\in\Lambda, n,r\in{\mathbb Z},
\grHom^n_{D^b(\sC_\gr)}(M,\nabla(\lambda)\langle r\rangle)\not=0\implies n=r.$$
Also, we can assume (after a refinement consistent with the original partial ordering) that the poset $\Lambda$ is
totally ordered.

 ($\Rightarrow$)  First, 
 $$\grHom^n_{D^b(\sC_\gr)}(\Delta(\lambda)\{a\},\nabla(\mu)\langle b\rangle)\cong\grHom^{n-a}_{D^b(\sC_\gr)}
 (\Delta(\lambda),\nabla(\mu)\langle b-a\rangle).$$
Hence, if the left hand side is non-zero, necessarily $\Ext^{n-a}_B(\Delta(\lambda),\nabla(\mu))\not=0$
so $n-a=0$ and $\lambda=\mu$ by well-known homological properties of standard and co-standard modules.
Hence, $\grHom_{\sC_\gr}(\Delta(\lambda),\nabla(\lambda)\langle b-a\rangle)\not=0$ so $b-a=0$.
Thus, $n-a=b-a$ so $n=b$, as required. Now it follows that $(\bigstar(M))$ holds if $M$ is a direct sum of
objects $\Delta(\lambda)\{m\}$, $m\in\mathbb Z$. Finally, if $N\to M\to Q\to$ is a distinguished triangle
in $D^b(\sC_\gr)$ and, if both $(\bigstar(N))$ and $(\bigstar(Q))$ hold, then $(\bigstar(M))$ holds since
$\grhom$ is a cohomological bifunctor (i.~e., takes distinguished triangles in either variable to long exact sequences). Thus, $(\bigstar(M))$ holds on $\sE^L$, as required.

 ($\Leftarrow$) 
  Consider the set $\Xi$ of ordered pairs $(|\Gamma|, m)$, where $\Gamma$ is an ideal (possibly the empty ideal) in $\Lambda$, and
  $m$ is a positive integer. The set $\Xi$ is ordered lexicographically. Given $X\in D^b(\sC_\gr)$, let 
  $d(X):=(|\Gamma|,m)$, where $\Gamma$ is the ideal generated by the maximal element $\gamma\in\Lambda$
  for which $[H^\bullet(X):L(\gamma)]\not=0$ and $m=[H^\bullet(X):L(\gamma)]$. If $X=0$, $d(X)=(0,0)$.

 Assume that $M\in D^b(\sC_\gr)$ satisfies the condition $(\bigstar(M))$
   We must show that $M\in\sE^L$. 
  We proceed by induction on $d(M)$ for $M$ satisfying $(\bigstar(M))$. If $d(M)=(0,0)$, then $M\cong 0$ and
  so $M\in\sE^L$, trivially. So assume that $d(M)\not=(0,0)$.   Let $d(M)=(|\Gamma|,m)$, and observe that
  because $\Lambda$ is totally ordered, the cardinality of $\Gamma$ determines $\Gamma$. Let
  $\gamma\in\Gamma$ be the unique maximal element. Since
  $H^\bullet(M)\in\sC_\gr[\Gamma]$, if $J=J(\Gamma)$, then
  $M$ belongs to the
  relative derived category $D^b_{\sC_{\gr}(B/J)}(\sC_\gr)$ which can 
  be identified with $D^b(\sC_\gr[\Gamma])$---see \cite[\S2,3]{CPS-1} and the references there.
  It suffices to show that $M\in\sE^L(B/J)$, since the natural full embedding $i_*:D^b(\sC_\gr(B/J))\to
  D^b(\sC_\gr(B))$ induced by the quotient map $B\twoheadrightarrow B/J$ carries $\sE^L(B/J)$
  to $\sE^L(B)$, by the inductive definition of $\sE^L(B/J)$.
  
  For some choice of integers $t$ and $r$,    $L(\gamma)\langle t\rangle$ is a composition factor
 of $H^r(M)$ in $\sC_\gr(B/J)$. 
  Because $\nabla(\gamma)\in B/J{\text{\rm --grmod}}$ is an injective module, 
$$0\not= \hom_{\sC_\gr}(H^r(M),\nabla(\gamma)\langle t\rangle)\cong\hom^{-r}_{D^b(\sC_\gr)}(M,\nabla(\gamma)
\langle t\rangle).$$
In particular, $t=-r$ since $(\bigstar(M))$ holds.
Choose a morphism $f:\Delta(\gamma)\{t\}\to M$ inducing a surjection
$$\grHom^t_{D^b(\sC_\gr)}(M,\nabla(\gamma)\langle t\rangle)\twoheadrightarrow 
\grHom^t_{D^b(\sC_\gr)}
(\Delta(\gamma)\{t\},\nabla(\gamma)\langle t\rangle)\cong k.$$
Consequently, for each integer $n$, there is a surjection
$$\grHom^n_{D^b(\sC_\gr)}(M,\nabla(\gamma)\langle t\rangle)\twoheadrightarrow 
\grHom^n_{D^b(\sC_\gr)}
(\Delta(\gamma)\{t\},\nabla(\gamma)\langle t\rangle)$$
(for $n\not= t$ both sides are 0).
Now form the distinguished triangle $\Delta(\gamma)\{r\}\to M\to M'\to$ and observe that
$$\grHom^n(M',\nabla(\lambda)\langle t\rangle)\subseteq\grHom^n(M,\nabla(\lambda)\langle t\rangle),
\quad \forall n,\lambda$$
 since $\grHom$ is a cohomological bifunctor.  In particular, this means that $(\bigstar(M'))$ holds.

Since $M'\in D^b(\sC_\gr[\Gamma])$ the composition factors $L(\gamma')$ of $H^\bullet(M')$ all satisfy
$\gamma'\in\Gamma$. However, 
$[H^n(M'):L(\gamma)]=[H^n(M):L(\gamma)]$ if $n\not=r$,  while
$[H^r(M'):L(\gamma)]<[H^r(M):L(\gamma)]$.  Thus, $d(M')<d(M)$. By induction, $M'\in\sE^L$. 
Since $\Delta(\gamma)\{r\}\in\sE^L$ as well, and $\Delta(\gamma)\{t\}\to M\to M'\to$ is distinguished, it
finally follows that $M\in\sE^L$.
\end{proof}
    
\begin{cor} Assume that $B$ is a standard Q-Koszul algebra with weight poset
$\Lambda$. For $\lambda,\mu\in\Lambda$, 
$$\grExt^n_B(\Delta^0(\lambda),\nabla_0(\mu)\langle r\rangle)\not=0\implies n=r.$$
Therefore, $B$ is a Q-Koszul algebra.
\end{cor}

\begin{proof} By Theorem \ref{better}, $\Delta^0(\lambda)\in \sE^L$ and $\nabla_0(\mu)\in\sE^R$. 
Using the definition of $\sE^L$, it is enough to check that 
$$\grHom^n(\Delta(\rho)\{a\},\nabla_0(\mu)\langle r\rangle)\not=0\implies n=r.$$
But $$\grHom^n(\Delta(\rho)\{a\},\nabla_0(\mu)\langle r\rangle)\cong \grHom^{n-a}(\Delta(\rho),
\nabla_0(\mu)\langle r-a\rangle),$$
so that $n-a=r-a$ or $n=r$ as required. A similar argument applies to $\sE^R$.\end{proof} 

We also have the following consequence of the above corollary together with  Theorems \ref{thirdprop}
and \ref{tensortensor}.

\begin{cor}\label{cor3.5} Assume that $B$ is a standard Q-Koszul algebra with weight poset $\Lambda$. Then
$B$ is a quadratic algebra. In addition, the $B\otimes_kB^\op$-modules $B_1$  and $B_1\otimes_{B_0}B_1$ each have a $\Delta^0\otimes_k\Delta^{0,\op}$-filtration, as does
the grade 2 relation module $W_2$ (defined above Theorem \ref{thirdprop}).\end{cor}

\medskip
\begin{center}{\bf Recollement}\end{center}

\smallskip
Finally, we indicate how recollement works for standard Q-Koszul algebras. More specifically, let
$B$ be standard Q-Koszul with poset $\Lambda$. Let $\Gamma$ be a non-empty ideal in $\Lambda$. 
Let $\Bmod[\Gamma]$ be the full subcategory of $\Bmod$ consisting of all finite dimensional modules
having composition factors $L(\gamma)$, $\gamma\in\Gamma$. Then $\Bmod[\gamma]\cong
\BJmod$, for some idempotent ideal $J$ in $B$. By \cite{CPS5}, $J=BeB$ for some idempotent $e\in
B_0$, so that $B/J$ is a positively graded quasi-hereditary algebra, and we can form the category
$\BJgrmod$ of graded $B/J$-modules. As remarked before, the quotient map $B\twoheadrightarrow B/J$ defines 
full embeddings $i_*:D^b(\BJmod)\to D^b(\Bmod)$ and $i_*:D^b(\BJgrmod)\to D^b(\Bgrmod)$. The
functor $i_*$ takes standard modules $\Delta^{B/J}(\gamma)$ for the quasi-hereditary algebra $B/J$ to
standard modules $\Delta^B(\gamma)=\Delta(\gamma)$, and we identify them accordingly. Similarly,
$i_*$ maps standard modules for $(B/J)_0$ to those for $B_0$, and we denote them both by
$\Delta^0(\gamma)$. Similar remarks apply to costandard modules. Since $i_*$ is a full embedding
at the derived category level, it preserves Ext-groups. Therefore, $B/J$ is also a standard Q-Koszul
algebra. 

Consider next the quasi-hereditary algebra $eBe$. The module category $eBe$-mod is equivalent to
the quotient category $\Bmod/\BJmod$, a highest weight category with weight poset $\Omega:=\Lambda\backslash\Gamma$. All this fits into a standard recollement diagram
$$ D^b(\BJmod) \begin{smallmatrix} {\overset {i^*}\longleftarrow }\\
 {\overset {i_*}\longrightarrow} \\
{\overset {i^!}\longleftarrow}\end{smallmatrix} D^b(\Bmod) \begin{smallmatrix}{\overset {j_!}\longleftarrow }\\ {\overset {j^*}\longrightarrow} \\
{\overset{ j_*}\longleftarrow}\end{smallmatrix}D^b(eBe{\text{\rm --mod}})$$
A similar recollement diagram is obtained by replacing $\Bmod$, $\BJmod$, $eBe-{\text{\rm mod}}$
by $\Bgrmod$, $\BJgrmod$, $eBe{\text{\rm --grmod}}$, respectively. It is well-known that given any
$\omega\in\Omega$, $j^*\Delta(\omega)$ (respectively, $j^*\nabla(\omega)$) is the standard (respectively,
costandard) module for $eBe$-mod attached to $\gamma$. Simiarly, $j^*L(\omega)$ is the
irreducible $eBe$-module attached to $\omega$. In addition, $j_!j^*\Delta(\omega)\cong \Delta(\omega)$
and $j_*j^*\nabla(\omega)\cong\nabla(\omega)$. The same holds true for $\Delta^0(\omega)$
and $\nabla^0(\omega)$. Consequently, we see that $eBe$ is itself a standard Q-Koszul algebra.

\section{$\Ext$-algebras}
First, we recall some recent results of Madsen \cite{Madsen}, \cite{Madsen2}, and we very briefly indicate some
problems suggested by these results in the context of Q-Koszul algebras. This topic will be discussed further in \cite{PS14}.

Let $A$ be a
positively graded finite dimensional algebra, and let $T$ be a finite dimensional tilting module
for the algebra $A_0$. Then $T$ is regarded as an  $A$-module through the natural map $A\to A_0$. Assume that
$A_0$ has finite global dimension. Following \cite{Madsen},  the algebra $A$ is defined to be $T$-Koszul provided that
$\grExt^j_A(T,T\langle i\rangle)\not=0\implies i=j.$ (Madsen \cite{Madsen} does not require that $A$ be finite dimensional, only finite
dimensional in each grade.)

Rather than define a tilting module $T$ for $A_0$,  we are interested here in the special case in which $A_0$ is a QHA with
weight poset $\Lambda$. In this situation,
$T=\bigoplus_{\lambda\in\Lambda}T(\lambda)^{\oplus n_\lambda}$, where the $n_\lambda$ are positive integers,
and $T(\lambda)$ is the unique (up to isomorphism) indecomposable $A_0$-module of highest weight $\lambda$ which has
both a $\Delta$- and a $\nabla$-filtration. In other words, $T$ is a full tilting module in the sense of quasi-hereditary algebras; cf. Ringel \cite{Ringel} and especially Donkin \cite{Donk}.

\begin{prop}\label{Mad} Assume that $A$ is a Q-Koszul algebra. Then $A$ is $T$-Koszul for any (full) tilting module $T$ for the
QHA $A_0$. \end{prop}

\begin{proof} Since $T$ has a $\Delta^0$-filtration and a $\nabla_0$-filtration, if $\grExt_A^j(T,T\langle i\rangle)\not=0$,
then, for some $\lambda,\mu\in\Lambda$, $\grExt_A^j(\Delta^0(\lambda),\nabla_0(\mu)\langle i\rangle)\not=0$, so
that, by (\ref{QKoszul}), $i=j$, as required.\end{proof}

In \cite[Thm. 4.2.1]{Madsen}, it is proved that if $A$ is any $T$-Koszul algebra, then
 $$A^\dagger:=\Ext^\bullet_A(T,T)^\op\quad{\text{\rm is $T^*$-Koszul.}}$$
Here $T^*:=\Hom_k(T,k)$, viewed as a left (tilting) module for $A_0^\dagger =\End_A(T,T)^\op\cong\End_{A_0}(T,T)^\op$.
In addition, $A\cong A^{\dagger\dagger}$. Moreover, if $A$ has finite global dimension (as is the case if $A$ is
a QHA), then \cite[Thm. 4.3.4]{Madsen}, applied with $A=\Gamma$ and $A^\dagger=\Lambda$ there,
gives an equivalence
$$G^b_{T^*}:D^b(A^\dagger{\text{\rm --grmod}})\overset\sim\longrightarrow D^b(A{\text{\rm --grmod}})$$
of triangulated categories.  

 The relationship of Q- and T-Koszulity needs to be understood better. On the one hand, the very ``Koszul-like" quadratic
 property proved in Theorem \ref{thirdprop}(a) seems to be unknown for general T-Koszul algebras, even those of finite global dimension. On the other hand,
 all current knowledge of $A^\dagger$ above currently derives from the T-Koszul property, which tells us, at present, only that
 $A^\dagger$ is $T^*$-Koszul. In particular, we do not
 know if $A^\dagger$ is Q-Koszul if $A$ is Q-Koszul, or even if $A$ is standard Q-Koszul. Formally, we ask the following
 questions.
 
\begin{ques}\label{Madsonques} Let $A$ be a Q-Koszul algebra.

\smallskip
(a) Under what conditions is the algebra $A^\dagger$ also Q-Koszul? 

\smallskip
(b) If $A$ is standard Q-Koszul, under what
conditions is $A^\dagger$ also standard Q-Koszul? 

\smallskip
(c) In those cases in which the answer to (b) is positive and $A$ has a Lie theoretic (or geometric) interpretation, is there
a corresponding interpretation for $A^\dagger$?  The classic example is the case of parabolic-singular duality in 
the category $\sO$ for a complex semisimple Lie algebra; if \cite{BGS}.

\end{ques}

\begin{rem}\label{TKos} 
Suppose that $A$ is standard Q-Koszul. Then using the methods of \cite{CPS1} or \cite{PS1a}, 
the following product formula can be deduced for $n\in\mathbb N$, $M\in \sE^L$,
and $N\in\sE^R$,
\begin{equation}\label{productformula}\begin{aligned}
\dim\Ext^n_{A}&(M,N)=\\
& \dim\Ext^n_A(M,N)=\sum_{a+b=n}\sum_{\nu\in\Lambda}\dim\Ext^a_A(\Delta(\nu),N)\cdot\dim\Ext^b_A(M,\nabla(\nu)). \end{aligned}
\end{equation}
In particular, the above equation holds for $M=\Delta^0(\lambda)$, $N=\nabla_0(\mu)$ for $\lambda,\mu\in\Lambda$. 

We do not prove (\ref{productformula}) here, but refer instead to the paper \cite{PS14} in preparation. 
(The methods are similar to those in \cite{CPS1} or \cite{PS1a}, working with enriched Grothendieck groups.)
The formula (\ref{productformula}) suggests that Question \ref{Madsonques})(b) has a positive answer without
any further conditions on $A$, i.~e., if $A$ is standard Q-Koszul, then $A^\dagger$ is always standard Q-Koszul
as well. This insight comes from \cite{CPS2}. There, conditions are satisfied, in the context of Kazhdan-Lusztig
theory, which guarantee that the homological dual of a (suitably structured) quasi-hereditary algebra is
again quasi-hereditary, or even possesses stronger properties. The proof involves a
formula like that in (\ref{productformula}) with $M$ and $N$ irreducible. \end{rem}

\medskip
\medskip
\begin{center}{\Large\bf Part II: An example in characteristic $p=2$}\end{center}

\medskip\medskip
\begin{center} {\bf Some further notation}\end{center}

\medskip The paragraphs below briefly describe the setup/notation employed in several previous papers \cite{CPS7},
\cite{PS9a}, \cite{PS9}, \cite{PS11}, \cite{PS10}, \cite{PS12}, \cite{PS13a}, and \cite{PS13}.  This material
will be used in the sections below.

\smallskip
Let $G$ be a semisimple, simply connected algebraic group over an algebraically closed field $k$ of positive characteristic $p$.   Fix a maximal torus $T$, contained in a Borel subgroup $B$ corresponding to the negative roots, etc. We follow the notation of \cite{JanB} carefully, except, given a dominant weight $\lambda\in X(T)_+$, 
$\Delta(\lambda)$ (respectively, $\nabla(\lambda)$) is the standard (respectively, cosandard) module of highest
weight $\lambda$. Thus, $\Delta(\lambda)$ (respectively, $\nabla(\lambda)$) has irreducible head (respectively, socle)
$L(\lambda)$. The set $X(T)_+$ is partially ordered by setting $\lambda\leq\mu$ provided $\mu-\lambda$ is
a sum of positive roots. We will work with ideals in $X(T)_+$ or in $X_{\text{\rm reg}}(T)_+$ (the set of $p$-regular
dominant weights, given its induced poset structure). If $\Gamma$ is a finite ideal, in either $X(T)_+$ or 
$X_{\text{\rm reg}}(T)_+$,  there is a QHA algebra $A=A_\Gamma$ such that $\Amod$ is equivalent to the category of finite dimensional rational $G$-modules which have
composition factors $L(\gamma)$, $\gamma\in\Gamma$. Indeed, we can assume $A$ is an appropriate quotient
algebra of the distribution algebra of $G$. 

The algebra $A_\Gamma$ can also be studied using the quantum enveloping algebra $U_\zeta$ at an
$\ell(p)$th root of unity associated to $G$.   Here $\ell(p)=p$ if $p$ is odd, and $\ell(2)=4$; see (\ref{LP}) below
in \S7. There is an appropriate $p$-modular system $(K,\sO,k)$ such that $U_\zeta$ is regarded as a
$K$-algebra. In addition, there is a (split) QHA quotient algebra $A'=A'_\Gamma$ of $U_\zeta$ such that $A'$--mod
is equivalent to the category of finite dimensional (type 1, integrable) $U_\zeta$-modules with composition factors
$L'(\gamma):=L_\zeta(\gamma)$, $\gamma\in\Gamma$. In addition, there is an $\sO$-order $\wA$ such that
$\wA_K\cong A'$ and $\wA_k\cong A$. 

Given $\lambda\in \Gamma$, the irreducible $A'$-module $L_\zeta(\lambda)$ of highest weight
$\lambda$ contains a minimal admissible lattice $\wL_{\text{\rm min}}(\lambda)$ and a maximal
admissible lattice $\wL_{\text{\rm max}}(\lambda)$. We set $\rDelta(\lambda):=\wL_{\text{\rm min}}(\lambda)_k$ and 
$\rnabla(\lambda) =\wL_{\text{\rm max}}(\lambda)_k$. These modules play an important role in the modular
representation theory of $G$. See \cite{CPS7} and the other references above for more discussion.

Define a positively graded $\sO$-order $\gr \wA=\bigoplus_{n\geq 0}\gr_n\wA$ by setting 
$$\gr_n\wA=\frac{\wA\cap\rad^nA'}{\wA\cap\rad^{n+1}A'},$$
where $\rad^nA'=(\rad A')^n$.\footnote{By an $\sO$-order (or simply an order if $\sO$ is clear) we mean an $\sO$-algebra $\wB$ which is a free $\sO$-module of finite rank. A $\wB$-module $\wM$ is a lattice, if it is free of finite rank
over $\sO$.} Similarly, if $\wM$ is a $\wA$-lattice, there is a graded $\gr\wA$-lattice
$\gr\wM:=\bigoplus_{n\geq 0} \gr_n\wM$, where $\gr_n\wM=(\wrad^n\wM)/(\wrad^{n+1}\wM)$, with $\wrad^n\wM:=
\wM\cap(\rad^n\wM_K)$. Here $\rad^n\wM_K=(\rad A')^n\wM_K$.  We can then define the (non-negatively) graded
algebra by setting
\begin{equation}\label{wgrA}\wgr A := (\gr\wA)_k.\end{equation}
In addition, for an $\wA$-lattice $\wM$, put $M:=\wM_k$ and $\wgr M:=(\gr \wM)_k$.

The papers cited above contain many properties of the algebras $\wgr A$ and their modules $\wgr M$, as well
as alternative definitions for them (usually involving the small quantum enveloping algebra).  Some of these
results will be mentioned in the final two sections of this paper.  

But it is important to observe that the graded algebra $\gr A:=\bigoplus_{n\geq 0}\rad^nA/\rad^{n+1}A$ is often 
different from the algebra $\wgr A$. The algebra $\wgr A$ is generally quasi-hereditary as is its grade 0
part $(\wgr A)_0$, with the same weight poset as $\wgr A$. Indeed, $(\wgr A)_0$ appears
to be highly worthy of further study.  See Remarks \ref{finremark}.\smallskip

The next two sections work with a slight variation of $\wgr A$, replacing
$A$ by a Schur algebra. (In fact, the Schur algebra could be placed in the current context, but we omit the
details. In addition, Schur algebras are more familiar to most readers.)

\medskip
\section{Turner's counterexample.}\label{turner} 
Given a semisimple, simply connected algebraic group $G$, the main result in \cite{PS11} establishes that any
standard module $\Delta(\lambda)$ has a $\rDelta$-filtration, provided the characteristic $p$ of the base field
$k$ is sufficiently large (depending on the root system of $G$). However, when $p$ is small this result sometimes
fails. In fact, an unpublished counterexample has been shown to us by Will Turner which involves the Schur algebra
$S(5,5)$ when $p=2$. In this section, we consider Turner's example for $S(5,5)$ in some detail. In \S6, we show that
despite the counterexample,  the modules $\rDelta(\lambda)$ do fit into an elegant standard Q-Koszul theory in the
case for $S(5,5)$ and $p=2$.

Specifically, there is a  ``forced graded"
version $\wgr S$ of  $S:=S(5,5)$. This graded algebra is obtained in the same way as the algebra $\wgr A$ above
the start of this section, but using the complex $q$-Schur
algebra $S':=S_{q}(5,5)$, with $q=-1$.\footnote{In the notation of \cite{PW} and \cite{Cliff}, $S'$ would
be denoted $S_{\sqrt{-1}}(5,5)$. In addition, $S'$ is a homomorphic image of the quantum enveloping
algebra for $\mathfrak{gl}_5({\mathbb C})$ at $\sqrt{-1}.$}  Then the modules $\rDelta(\lambda)$ are the standard
modules for $(\wgr S)_0$, a quasi-hereditary quotient algebra of $S$ itself. The main result in \S6, which is
built on the results of this section, shows that
$\wgr S$ is standard Q-Koszul. The authors regard this highly non-trivial result in the smallest possible
characteristic as quite remarkable. Together with the large prime results mentioned in \S7, 
it inspires the conjectures given there.

The discussion requires some standard partition terminology. For a
positive integer $r$, let $\Lambda^+(r)$  (respectively, $\Lambda(r)$)
be the set of partitions (compositions) of $r$ with at most
$r$ nonzero parts.  Let $\Lamreg(r)\subset\Lambda^+(r)$ be the
$2$-regular partitions (i.~e., $\lambda\in\Lambda_{\text{\rm reg}}^+(r) \iff$ no part
of $\lambda$ is repeated $2$ or more times).\footnote{The 2-regular partitions
should not be confused with the set of $2$-regular weights in the sense of
alcove geometry.} If $\lambda\in\Lambda(5)$, $\lambda^\star$ is the dual
partition.

Let $k$ be an algebraically closed field of characteristic 2. Let $R=k{\mathfrak S}_5$ be the group algebra of the symmetric group
${\mathfrak S}_5$ over $k$. (This $R$ is obviously not to be confused with the root system which has the same
name.) For $\lambda\in\Lambda(5)$, let
${\mathfrak S}_\lambda$ be the Young subgroup of ${\mathfrak S}_5$
defined by $\lambda$. The Poincar\'e polynomial of ${\mathfrak
S}_\lambda$ is defined by $p_{{\mathfrak
S}_\lambda}(q)=\sum_{w\in{\mathfrak S}_\lambda}q^{\ell(w)}$. Taking
$\lambda=(5)$, ${\mathfrak S}_\lambda={{\mathfrak S}_5}$ has
Poincar\'e polynomial
$$ p_{{\mathfrak S}_5}(q)=\prod_{i=1}^5\frac{q^i-1}{q-1}.$$
In particular, we will need the fact that 
\begin{equation}\label{poincare}
r_{(3,2)}(q):=p_{{\mathfrak S}_5}(q)/p_{{\mathfrak S}_{(3,2)}}(q)=
(1+q^2)(1+q+q^2+q^3+q^4).\end{equation}
 Let ${^\lambda}{\mathfrak S}_5$ denote the set of distinguished
 right coset representatives of ${\mathfrak S}_\lambda$ in
 ${\mathfrak S}_5$, In particular, $r_{(3,2)}(q)=\sum_{d\in{^{(3,2)}{\mathfrak
 S}_5}}q^{\ell(d)}$.

Set $T_\lambda:=\ind_{{\mathfrak S}_\lambda}^{{\mathfrak S}_5}k$,
the right permutation module for ${\mathfrak S}_5$ acting on the set
of right cosets of the subgroup ${\mathfrak S}_\lambda$. If
$T=\bigoplus_{\lambda\in\Lambda(5)}T_\lambda$,  
$$S(5,5):=\End_R(T)$$
is. by definition, the 
Schur algebra of bidegree $(5,5)$ over $k$; see \cite{Green}. The category $S(5,5)$-mod
is a highest weight category with weight poset $\Lambda^+(5)$. For
$\lambda\in\Lambda^+(5)$, let $\Delta(\lambda)$ (respectively, $\nabla(\lambda)$
$L(\lambda)$) be the standard (respectively, costandard, irreducible) $S(5,5)$-module
indexed by $\lambda$.

Let $\modR$ be the category of finite dimensional right $R$-modules. (Recall that $R=k{\mathfrak S}_5$.) It is related to
the category $S$-mod of finite dimensional left $S$-modules by the contravariant {\it diamond
functors}:
$$\begin{cases} (-)^\diamond=\Hom_S(-,T):S{\text{\rm --mod}}\longrightarrow
\modR;\\
(-)^\diamond=\Hom_R(-,T):\modR\longrightarrow\Smod.\end{cases}
$$
For $\lambda\in\Lambda^+(5)$, $\Delta(\lambda)^\diamond\cong
S_\lambda$, the Specht module for $R$ indexed by
$\lambda$. The irreducible $R$-modules are indexed by the set of
$\Lamreg(5)$ of $2$-regular partitions; given
$\lambda\in\Lamreg(5)$, $D_\lambda$ denotes the associated
irreducible module. For $\lambda\in\Lambda^+(5)$,
\begin{equation}\label{irreduciblediamond}L(\lambda)^\diamond\cong\begin{cases} D_{\lambda^\star},
\,\,\lambda^\star\in\Lamreg(5);\\ 0,\,\,{\text{\rm
otherwise.}}\end{cases}\end{equation}
 (We remark that the description of $L(\lambda)^\diamond$ requires a twist by the
 sign representation in characteristics different from 2.)
  Also, for $\lambda,\mu\in\Lambda^+(5)$,
$L(\lambda)$ and $L(\mu)$
 are in the same block if and only if the partitions $\lambda,\mu$
 have the same $2$-core.

Because $_ST=R^\diamond$ and $R_R$ is a direct summand of $T_R$, it
follows that $_ST$ is a projective $S$-module. It is known that
$_ST$ is self-dual, so that $T$ is also an injective $S$-module (and hence a tilting module) and
the functor $(-)^\diamond=\Hom_S(-,T)$ is exact.

For $\lambda\in\Lambda^+(5)$, let $X(\lambda)$ be the tilting module
for $S$ defined by $\lambda$. It has a $\Delta$-filtration with
bottom section $\Delta(\lambda)$ and higher sections $\Delta(\mu)$
for partitions $\mu<\lambda$ (in the dominance ordering) having the
same $2$-core as $\lambda$.
\footnote{More precisely, this observation can be quickly
reduced to the case of Young modules for symmetric groups, where it is well-known.}  

A PIM $Y$ in $R_R$ has irreducible socle $D_{\mu^\star}$ with $\mu^\star\in\Lambda^+(5)$. Thus,
the corresponding summand $Y^\diamond$ of $T$ is the projective cover $P(\mu)$ of $L(\mu)$. If $\nu\in\Lambda^+(5)$, then
$$[P(\mu):\Delta(\nu)]=[Y:S_\nu]=[S_\nu:D_{\mu^\star}].$$
Taking $\mu=(2^2,1)$, $\dim\,D_{\mu^\star}=4$, and $D_{\mu^\star}$ is the unique non-trivial principal
block irreducible module. It follows easily that $P(2^2,1)$ is filtered by standard modules $\Delta(\nu)$ with
$\nu=(2^2,1), (3,1^2),$ and $(3,2)$, each appearing with multiplicity 1. Since we know this PIM for $S$ is an indecomposable tilting module, it must be $X(3,2)$. 
Since $X(3,2)=P(2^2,1)$ has simple head $L(2^2,1)$ and is 
self dual, $X(3,2)$ has head and socle isomorphic to $L(2^2,1)$. In
particular, this means that 
\begin{equation}\label{sox} {\text{\rm $L(2^2,1)$ is the socle of
$\Delta(3,2)$.}}\end{equation}

On the other hand, $\rDelta(1^5)\cong L(1^5)$,
$\rDelta(3,2)\cong L(3,2)$, and $\rDelta(3,1,1)\cong L(3,1,1)$.
(The last two isomorphisms are obtained by computing dimensions using versions of
Steinberg's tensor product theorem; see the table below.)

\medskip\noindent{Claim:} $\rDelta(2^2,1)=\Delta(2^2,1)\not\cong L(2^2,1)$.
(Assuming this fact, it follows that $\Delta(3,2)$ does not have a
$\rDelta$-filtration.)

To check the claim, it will be necessary to use the
description of the $\rDelta$-modules from the quantum point of view.
This uses the theory of $q$-Schur algebras as well as
their relationship to Hecke algebras by means of quantum Schur-Weyl
duality \cite{PS1}, \cite{DPS1}.

Let $R'=H({\mathfrak S}_5)$ be the Hecke algebra over $\mathbb C$ of
${\mathfrak S}_5$ with $q=-1$. It has standard basis $\tau_w$,
$w\in{\mathfrak S}_5$, satisfying the familiar relations
\cite[(1.1)]{DPS1}. For $\lambda\in\Lambda(5)$, let
$x_\lambda=\sum_{w\in {\mathfrak S}_\lambda}\tau_w$ and set
$T'_\lambda= x_\lambda R'$, the $q$-permutation module defined by
$\lambda$. In particular, if $\lambda=(5)$, $T'_{(5)}={\mathbb
C}x_{(5)}$ is the one-dimensional index representation of $R'$. We
have, for any $\lambda$,
$$\dim\,\Hom_{R'}(T'_\lambda,T'_{(5)})=\dim\,\Hom_{R'}(T'_{(5)},T'_\lambda)=1,$$
by Frobenius reciprocity and the fact that the $q$-permutation
modules are self-dual.

In particular, let $\lambda=(3,2)$, and consider the nonzero
homomorphisms
$$\phi:T'_{(5)}\to T'_{(3,2)}, \quad x_{(5)}\mapsto
\sum_{d\in{^{(3,2)}{\mathfrak S}_5}}x_{(3,2)}\tau_d$$ and
 $$\psi:T'_{(3,2)}\to T'_{(5)},\quad x_{(3,2)}h\mapsto x_{(5)}h.$$
 Using (\ref{poincare}),
 $$
 \begin{aligned} \psi\circ\phi  &= \sum_{d\in{^{(3,2)}}{\mathfrak
 S}_5}q^{\ell(d)}x_{(5)}\\
 &=r_{(3,2)}(-1)x_{(5)} = 2x_{(5)}.\end{aligned}$$

 For $\lambda\in\Lambda^+(5)$, let $S'_\lambda$ (respectively, $Y'_\lambda$, $Y^{\prime\natural}_\lambda$)
be the corresponding Specht (respectively, Young, twisted Young module)
module; see \cite[\S2]{PS1}. Then $Y^{\prime\natural}_\lambda\cong
{Y'_{\lambda'}}^\Phi$, using the notation of \cite[(2.0.4)]{PS1}
with the involution $\Phi:R'\to R'$ defined in \cite[(2.0.3)]{PS1}.

Putting $T'=\bigoplus_{\lambda\in\Lambda(5)} T'_\lambda$, let
$$S'=S_q(5,5)=S_{-1}(5,5):=\End_{R'}(T')$$
 be the $q$-Schur algebra of  bidegree $(5,5)$ at
$q=-1$. The algebra $S'$ is quasi-hereditary with weight poset
$\Lambda^+(5)$, standard objects $\Delta'(\lambda)$ and irreducible
objects $L'(\lambda)$ for $\lambda\in\Lambda^+(5)$.

 Consider the tilting module $X'(2^2,1)$ for $S'$ corresponding to the
 partition $(2^2,1)$. Then
 $$X'(2^2,1)^\diamond \cong Y_{(3,2)}^{\prime\Phi}.$$
(To see this, one can use \cite[Prop. 7.3(d)]{DPS1} and \cite[Lemma
1.5.2]{CPSMem}.)  Now the possible $\Delta'$ sections of $X'(2^2,1)$
are $\Delta'(2^2,1)$ (with multiplicity 1) and $\Delta'(1^5)$ of an
undetermined multiplicity. Therefore, $Y_{(3,2)}^{\prime\Phi}$ has a
filtration with sections $S'_{(2^2,1)}=\Delta'(2^2,1)^\diamond$ and
$S'_{(1^5)}={\Delta'}(1^5)^\diamond$. Since ${S'_\lambda}^\Phi$ is
dual to $S'_{\lambda'}$ and since $Y'_\lambda$ is self-dual, it
follows that $Y'_{(3,2)}$ has a filtration with sections
$S'_{(3,2)}$ (with multiplicity 1) and possibly $S'_{(5)}$ (having the same
multiplicity as $\Delta'(1^5)$ does in $X'(2^2,1)$). But
$Y'_{(3,2)}$ is an indecomposable summand of $T'_{(3,2)}$. We have
already proved that any nonzero homomorphism $T'_{(3,2)}\to
T'_{(5)}$ or $T'_{(5)}\to T'_{(3,2)}$ splits. It follows that
\begin{equation}\label{Y'_(3,2)} Y'_{(3,2)}\cong S'_{(3,2)}.\end{equation}
Therefore, $X'(2^2,1)\cong
\Delta'(2^2,1)$. But $X'(2^2,1)$ is self-dual, so that
\begin{equation}\label{Delta'(2^2,1)}
\Delta'(2^2,1)\cong L'(2^2,1).\end{equation}
This forces $\rDelta(2^2,1)\cong
\Delta(2^2,1)$. To finish the Claim, it must be checked that
$\Delta(2^2,1)$ is not irreducible. Otherwise, $S_{(2^2,1)}=\Delta(2^2,1)^\diamond$
is irreducible of dimension 5. But the only possible dimensions of irreducible 
${\mathfrak S}_5$-modules in characteristic 2 are 1 and 4; see Carlson \cite{Car}, for example.
This completes the proof of the following result.

\begin{prop} (Turner) For the Schur algebra $S(5,5)$ in characteristic 2, the standard module
$\Delta(3,2)$ does not have a $\rDelta$-filtration.\end{prop}

\section{Continuation: $\wgr S(5,5)$ is a standard Q-Koszul algebra in characteristic 2.}

We continue our discussion of $S(5,5)$, focusing on its principal block
$A$.\footnote{There is only one other block which is easily handled by ad hoc methods.} The partitions associated
to its irreducible modules form a poset $\Lambda$, given by the first column of the table below.

Only the trivial module $D_{(5)}$ for $k{\mathfrak S}_5$ has dimension 1, so the other modules all have
dimension 4. Applying (5.0.13) and the fact that $S_{(2^2,1)}$ has dimension 5,
gives that $\Delta(2^2,1)$ has two composition
factors $L(2^2,1)$ and $L(1^5)$, each occurring with multiplicity 1.
The additional information in the following table can be readily checked using the Weyl dimension formula
and the  Steinberg tensor product
theorem (both the characteristic 2 version and the quantum version \cite{Cliff}).  As mentioned above, the  modules $L(\lambda)$ listed are precisely those in the
``principal" block for $S(5,5)$ in characteristic 2 (associated to the determinant representation
$\Delta(1^5)=L(1^5)$). Dimensions of the corresponding irreducible modules for the characteristic 0 quantum
$q$-Schur algebra, $q=-1$, are also given. We denote the ``principal block" of this $q$-Schur algebra by $A'$, and 
generally decorate with the ``prime" symbol objects associated to $A'$.

\medskip
\begin{center}\begin{tabular}{ | l || r | r | r | }
\hline
$\lambda$ &  $\dim\Delta(\lambda)=\dim\Delta'(\lambda)$ & $\dim L'(\lambda)$ & $\dim L(\lambda)$ \\ \hline
\hline
($1^5$) & 1 & 1 & 1 \\ \hline
($2^2,1$) & 75 & 75 & 74\\ \hline
($3,1^2$) & 126 & 50 & 50 \\ \hline
($3,2$) & 175 & 50 & 50
\\ \hline
($5$) & 126 & 75 & 25 \\ \hline
\end{tabular}\end{center}
\begin{center} Various dimensions of irreducible and standard modules for $S=S(5,5)$ in characteristic 2 and
$S'=S_{-1}(5,5)$, in characteristic 0
\end{center}

\medskip\medskip
\medskip\medskip
Now we give the matrix $D$ of decomposition numbers $[\Delta(\lambda):L(\mu)]$. The entries $x$ and $y$ (which will be shown shortly to  be equal to 1) are non-negative integer values, momentarily unknown, 
with $x+y=2$. In this matrix, all entries and the constrains on $x$ and $y$ can can be determined
solely from

\begin{enumerate}
\item[(1)] The previous table;
\item[(2)] The fact that $[\Delta(\lambda):L(\mu)]\not=0$ implies $\mu\leq\lambda$;
\item[(3)] $[\Delta(\lambda):L(\lambda)]=1$;
\item[(4)]   $[\Delta(\lambda):L(1^5)]=[S_\lambda:D_{(5)}]\leq\dim S_\lambda\leq 6$.
\end{enumerate}
 
 \medskip\medskip
\begin{center}
\begin{tabular}{|c || c| c| c| c| c|}
\hline
& $L(1^5)$ & $L(2^2,1)$ & $L(3, 1^2)$ & $L(3,2)$ & $L(5)$ \\ \hline \hline
$\Delta(1^5)$ & 1 & 0 &0 & 0 & 0 \\ \hline
$\Delta(2^2,1)$ & 1 & 1 & 0 &0 & 0\\ \hline
$\Delta(3,1^2)$ & 2 & 1 &1 & 0 & 0\\ \hline
$\Delta (3,2)$ & 1 &1&1&1 & 0 \\ \hline
$\Delta(5)$ & 1 & 0 & $x$ &  $y$
& 1 \\ \hline
\end{tabular}\end{center}
 \begin{center} Decomposition matrix $D$\end{center}

\medskip
The first row in the Cartan matrix $C=D^t\cdot D$ has entries $8,4, 3+x, 1+y,1$. No row in $C$, other than the first row, can have three entries as large as 4 (given the constraint that $x+y= 2$). However, the
Cartan matrix corresponding to the PIMs,  given in \cite{Car} for a block in the algebra $H$ Morita equivalent to
$S(5,5)$, has a row with entries 8,4,4,2,1
(in some order).   For the reader's convenience, this Cartan matrix is given below, with rows and columns
as indexed in \cite{Car}, but in a different order.

Thus, the first row computed above must correspond to the unique row with these
entries in \cite{Car}. Comparison of rows forces $x=1$. Thus, $y=1$.  At this point, the matrix
$C=D^t\cdot D$ agrees with that in \cite{Car} (after a simultaneous reordering of rows and columns) as listed
below. No further simultaneous reordering of rows and columns leads to the same $5\times 5$ matrix.
Since this matrix is $C$ above, the conversion table below, of partitions to labels in \cite{Car}, is uniquely determined.

\begin{center}
\begin{tabular}{|c ||c | c| c| c| c|}\hline \hline
& $L(7)$ & $L(2)$ & $L(6)$ & $L(5)$ & $L(4)$ \\ \hline
$P(7)$ & 8 &4&4&2&1\\ \hline
$P(2)$& 4& 3& 2 & 1& 0\\ \hline
$P(6)$ & 4& 2 & 3 & 2 &1\\ \hline
$P(5)$ & 2 & 1 &2 & 2 & 1\\ \hline
$P(4)$ &1 & 0 & 1 & 1 & 1\\ \hline
\end{tabular}\end{center}

\begin{center} Cartan matrix $C$ for the principal block for $S(5,5)$ in characteristic 2, labeling as in \cite{Car}\end{center}

\medskip\medskip\begin{center}
\begin{tabular}{|c|c|}\hline
$(1^5)$ & 7\\ \hline
$(2^2,1)$ & 2\\ \hline
$(3,1^2)$ & 6\\ \hline
$(3,2)$ & 5\\ \hline
$(5)$ & 4\\ \hline
\end{tabular}\end{center}
\begin{center} Conversion table from partitions to labeling in \cite{Car}\end{center}
\medskip

The decomposition matrix $D'$ for the corresponding block of $S'$ can be easily obtained using
entries from $D$, the equality $\Delta'(2^2,1)=L'(2^2,1)$, and (5.0.12). It is given below, 
 indexing these modules with the integer labels above.
In this terminology, (5.0.12) implies that $[\Delta'(4):L'(5)]=[P'(5):\Delta'(4)]=0$. 
\medskip\medskip

\begin{center}
\begin{tabular}{|c||c|c|c|c|c|}
\hline
& $L'(7)$ & $L'(2)$ & $L'(6)$ & $L'(5)$ & $L'(4)$\\ \hline\hline
$\Delta'(7)$ & 1 & 0 &0 &0 &0\\ \hline
$\Delta'(2)$ & 0&1 &0&0&0\\ \hline
$\Delta'(6)$ & 1&1&1&0&0\\ \hline
$\Delta'(5)$ & 0 &1&1&1&0\\ \hline
$\Delta'(4)$ & 1 &0&1&0&1\\ \hline
\end{tabular}\end{center}
\begin{center} Decomposition matrix $D'$ for principal block for $S_{-1}(5,5)$ in characteristic 0\end{center}

\medskip
We now describe the radical and socle series
for the PIMs corresponding to the irreducible modules $L(7), L(2), L(6), L(5)$ and $L(4)$, as given in \cite{Car}.
\begin{center}
\begin{tabular}{|c|c|c|c|c|}
\hline
$P(7)$ & $P(2)$ &$P(6)$ &$P(5)$ &$P(4)$\\ \hline
\hline
$L(7)$ & $L(2)$ & $L(6)$ & $L(5)$ & $L(4)$\\ \hline
$L(2), L(6)$ & $L(7)$ & $L(5), L(7)$ & $L(4)$, $L(6)$ &$L(5)$\\ \hline
$L(5), L(7), L(7)$ & $L(6)$ & $L(2), L(4)$ & $L(5), L(7)$& $L(6)$\\ \hline
$L(2), L(4), L(6)$ & $L(5), L(7)$ & $L(5), L(7)$ & $L(2), L(6)$ & $L(7)$\\ \hline
$L(5), L(7), L(7)$ & $L(2)$ & $L(6), L(6)$ & $L(7)$ & \\ \hline
$L(2), L(6)$ & $L(7)$&  $L(7)$ &  &\\ \hline
$L(7), L(7)$& $L(6)$ & $L(2)$ & & \\ \hline
$L(2), L(6)$ & $L(7)$ & $L(7)$ & & \\ \hline
$L(7)$ & $L(2)$ & & &  \\ \hline
\end{tabular}\end{center}

\medskip
\begin{center} Radical series for PIMs $P(7), P(2), P(6), P(5), P(4)$ \end{center}

\medskip\medskip
\begin{rem}\label{rem4.2} We record the interesting facts that the radical series table above
shows that $\gr S(5,5)$, the graded algebra obtained from $S(5,5)$ by grading it through its radical series
filtration, is neither Koszul nor quasi-hereditary.  First, suppose that $\gr S(5,5)$ is Koszul and consider
the minimal projective resolution of $L(2)$. It begins as 
$\gr P(7)\overset\alpha\longrightarrow P(2)\twoheadrightarrow L(2)$. From the table immediately above, the
kernel of $\alpha$ must be an image of $\gr P(2)$ and must also contain $L(4)$ as a composition factor. But $L(4)$
does not appear as a composition factor of $P(2)$, a contradiction.  Hence, $\gr S(5,5)$ is not Koszul. Secondly, suppose that $\gr S(5,5)$ is QHA. Except for $\gr P(4)$, the head of each graded PIM occurs with multiplicity $>1$ in the PIM. 
It follows that $\gr P(4)$ is a standard module with head $L(4)$. Also, the weight ``4" is maximal. Obviously,  $\dim\Hom_{\gr S(5,5)}(\gr P(4),\gr P(6))=1$.  In particular, using the maximality of ``4", the standard module
$\gr P(4)$ must appear with multiplicity 1 in a standard module filtration of $\gr P(6)$, and may be taken to occur at the bottom, as a submodule. Also, by (\ref{gradedExt}), 
$$\dim\grHom_{\gr S(5,5)}(\gr P(4)\langle m\rangle,\gr P(6))=1, \quad{\text{\rm for a unique $m$.}}$$
 Necessarily, $m=2$. The resulting graded
map must be an injection, since its ungraded version is an injection. Thus, $L(7)\langle 5\rangle$ lies in the socle of $\gr P(6)$. 
The table above shows that $\gr P(6)$, viewed as a graded $(\gr A)_0$-module, has a unique composition
 factor $L(7)\langle 0\rangle$, and it occurs in---and exactly as---the graded $(\gr A)_0$-submodule 
 $\gr P(6)_5$. It follows that the latter $(\gr A)_0$-submodle is contained in the $\gr A$-socle of $\gr P(6)$. So, 
 $0=(\gr A)_1(\gr P(6))_5=(\gr P(6))_6\not=0$, a
contradiction. Of course, the algebra $S(5,5)$ is quasi-hereditary, so the above discussion shows that, in general,
the quasi-hereditary property is not preserved under
 ``passing to the radical series grading." In addition, $S(5,5)$ is not
Koszul. Otherwise, $\gr S(5,5)\cong S(5,5)$, a general property of Koszul algebas. However, this isomorphism
gives two contradictions, since $\gr S(5,5)$ has been shown to be neither Koszul nor, unlike $S(5,5)$, quasi-herediitary. \end{rem}

\begin{center}\begin{tabular}{ |c|c|c|c|c|}
\hline
$P(7)$ & $P(2)$ &$P(6)$ &$P(5)$ &$P(4)$\\ 
\hline\hline
$L(7)$ & $L(2)$ & $L(6)$ & $L(5)$ & $L(4)$ \\ \hline
$L(2), L(6)$& $L(7)$ & $L(7)$ & $L(4), L(6)$ & $L(5)$ \\ \hline
$L(7), L(7)$ & $L(6)$ & $L(2)$ & $ L(5), L(7)$& $L(6)$ \\ \hline
$L(2), L(6)$ & $L(7)$ & $L(5),L(7)$ & $L(2), L(6)$ & $L(7)$\\ \hline
$L(5), L(7), L(7)$ & $L(2)$ & $L(4), L(6)$ & $L(7)$ & \\ \hline
$L(2), L(4), L(6)$ & $L(5), L(7)$ & $L(5), L(7)$ & & \\ \hline
$L(5), L(7), L(7)$ & $L(6)$& $L(2), L(6)$ & & \\ \hline
$L(2), L(6)$ & $L(7)$ & $L(7)$ & &\\ \hline
$L(7)$ & $L(2)$ &&&\\ \hline
\end{tabular}
\end{center}
\begin{center} Socle series for PIMS $P(7), P(2), P(6), P(5), P(4)$\end{center}

\medskip
Also, we have the following radical series for the standard modules. It is the same as the socle series.

\medskip
\begin{center}\begin{tabular}{| c| c| c| c| c|}
\hline
$\Delta(7)$ &$ \Delta(2)$ & $\Delta(6)$ & $\Delta(5)$ & $\Delta(4)$ \\ \hline\hline
$L(7) $ & $L(2)$ & $L(6)$ & $L(5)$ &$L(4)$  \\ \hline
 & $L(7)$ & $L(7)$ & $L(6)$ & $L(5)$\\ \hline
 && $L(2)$ & $L(7)$ & $L(6)$\\ \hline
 && $L(7)$ & $L(2)$ & $L(7)$\\ \hline
 \end{tabular} \end{center}
 \medskip\begin{center} The radical/socle series for standard modules for $S(5,5)$\end{center}
 
 \medskip
 The first and second columns are clear from the decomposition matrix $D$ above. The column
 for $\Delta(6)$ follows by inspecting first the socle series for $P(6)$ and then its radical series.
 Standard quasi-hereditary theory says that $\Delta(6)$ is a quotient of $P(6)$, and is the unique
 quotient with the composition factors of $\Delta(6)$ (counting multiplicities). Also, $\Delta(4)=P(4)$
 Now consider $\Delta(5)$. Finally, $\Delta(5)\cong P(5)/\Delta(4)$. By (\ref{sox}), $\Delta(5)$
 has socle $L(2)$. (Note also (from the radical series of $P(7)$) that is no non-trivial extension between $L(4)$ and $L(7)$.  It follows that the description of $\Delta(5)$ is as indicated.
 
 \medskip
 
 Now we begin to discuss the quantum case. First, the table below describes the radical series for the
 quantum standard modules $\Delta'(\lambda)$.
 \medskip
 
 \begin{center}\begin{tabular}{| c| c| c| c| c|}
\hline
$\Delta'(7)$ &$ \Delta'(2)$ & $\Delta'(6)$ & $\Delta'(5)$ & $\Delta'(4)$ \\ \hline\hline
$L'(7) $ & $L'(2)$ & $L'(6)$ & $L'(5)$ &$L'(4)$  \\ \hline
 & & $L'(7)$, $L'(2)$ &$L'(6)$& $L'(6)$\\ \hline
 &&& $L'(2)$ & $L'(7)$ \\ \hline
\end{tabular} \end{center}

 \smallskip
 \begin{center} Radical=socle series for quantum standard modules for $S_{-1}(5,5)$ in characteristic 0
 \end{center}
 
 To see this, first note that
 $\Delta'(7)=L'(7)$ and $\Delta'(2)=L'(2)$ from the decomposition matrix. Suppose that $L'(\lambda)$ is a submodule of $\Delta'(\mu)$. Then 
 $L'(\lambda)\cap\wDelta(\mu)$ is both a full lattice in $L'(\lambda)$ and a pure submodule of $\wDelta(\mu)$. Thus,
 some composition factor of $\rDelta(\lambda)$ must appear in the socle of $\Delta(\mu)$.  Consequently, the
 socle of $\Delta'(4)$ must be $L'(7)$ and the socle of $\Delta'(5)$ must be $L'(2)$.  Using also
  the
 quantum decomposition matrix, we get the columns for $\Delta'(5)$ and $\Delta'(4)$. From this information,
 $\Ext^1(L'(6),L'(2))\not=0\not=\Ext^1(L'(6),L'(7))$ so the middle column in the table follows.
 
 We next describe the $\rDelta$-modules in the following table which gives their radical and socle series. The 
 $\rnabla$ are described by turning the diagrams upside down.  
 
  \begin{center}\begin{tabular}{| c| c| c| c| c|}
\hline
$\rDelta(7)$ &$ \rDelta(2)$ & $\rDelta(6)$ & $\rDelta(5)$ & $\rDelta(4)$ \\ \hline\hline
$L(7) $ & $L(2)$ & $L(6)$ & $L(5)$ &$L(4)$  \\ \hline
 &  $L(7)$& && $L(5)$\\ \hline
\end{tabular} \end{center}

 Before giving the next table, note there is a natural ``Loewy index" $\ell(P'(\lambda),\Delta'(\mu))$ that can be
 defined for any $\Delta$-filtration section $\Delta(\mu)$ of a PIM $P(\lambda)$.
 for a quasi-hereditary algebra with weight poset $\Lambda$.  For simplicity, we consider only the case when $\Delta(\mu)$ appears with multiplicity one as a section of $P(\lambda)$.\footnote{More generally, define
 $\ell(P(\lambda),\Delta(\mu))$ to be one less than the maximum Lowey length of $f(P(\lambda))$, with
 $f$ ranging over all $f\in\Hom(P(\lambda),\nabla(\mu))$ (or just over a basis of the latter space). The definition can be used to define $\ell(M,\Delta(\mu))$ for any finite dimensional module $M$ for the underlying quasi-hereditary algebra.}
 Namely, using the ``prime" notation here, extend the natural
 map $\Delta'(\mu)\to\nabla'(\mu)$ to a map $f:P'(\lambda)\to\nabla'(\mu)$.  The multiplicity one assumption
 guarantees uniqueness of such an extension. Now define $\ell(P'(\lambda),\nabla'(\mu))$ to be 
 the Loewy length of $f(P'(\lambda))-1$. In the table below these Loewy lengths are indicated in the
 left hand column. They can be computed
 using the previous table and the natural duality on $S'(5,5)$. 
  
 \medskip
 
 \begin{center}
\begin{tabular}{|c ||c |c |c |c |c |}\hline

& $P'(7)$ & $P'(2)$ & $P'(6)$ & $P'(5)$ & $P'(4)$ \\ 
\hline\hline $0$ & $\Delta'(7)$ & $\Delta'(2)$ &  $\Delta'(6)$&        $\Delta'(5)$  & $\Delta'(4)$       
  \\ \hline
$1$ & $\Delta'(6)$ & $\Delta'(6)$      & $\Delta'(5), \Delta'(4)$ &         &        \\ \hline
$2$ & $\Delta'(4)$ & $\Delta'(5)$   &  &  &                        \\ \hline
\end{tabular}
\end{center}

 \medskip
 According to \cite{SVV}, the algebra $A'$ is standard Koszul.\footnote{A Koszul algebra is standard Koszul if it is
 quasi-hereditary and if its standard modules have a ``linear" projective resolution. Linear here means that the
 terms in cohomological degree $-n$ are generated in grade $n$. See \cite{Maz}.} In particular, $\gr A'\cong A'$ is a quasi-hereditary algebra with $\gr P'(\lambda)\cong P'(\lambda)$ and $\gr\Delta'(\lambda)\cong\Delta'(\lambda)$
 as graded modules.  Each PIM $\gr P'(\lambda)$ has a filtration by shifted standard modules $\gr\Delta'(\mu)\langle s\rangle$, for $s\geq 0$. The multiplicity 
 $$[\gr P'(\lambda):\gr\Delta'(\mu)\langle s\rangle]=\dim \grHom_{\gr A'}(\gr P'(\lambda),\gr^\circ\nabla'(\mu)\langle s\rangle).$$
  That is, this
 number is precisely the multiplicities of $L'(\lambda)$ in the $-s$th socle layer of $\gr^\circ\nabla(\mu)$,
 or equivalently, in the $(-s)$th socle layer of $\nabla'(\mu)$ itself. In our case, where $[P'(\lambda):\Delta'(\mu)]
 \leq 1$, this multiplicity in 1 if $\ell(P'(\lambda),\Delta'(\mu))=s$, and 0 otherwise. Thus, the table above
 may be reinterpreted as giving the required graded multiplicities. We repeat it for emphasis, with the
 left hand column now giving graded multiplicity information.

  \medskip
 \begin{center}
\begin{tabular}{|c ||c |c |c |c |c |}\hline

& $\gr P'(7)$ & $\gr P'(2)$ & $\gr P'(6)$ & $\gr P'(5)$ & $\gr P'(4)$ \\ 
\hline\hline $0$ & $\gr \Delta'(7)$ & $\gr \Delta'(2)$ &  $\gr \Delta'(6)$&        $\gr \Delta'(5)$  & $\gr\Delta'(4)$       
  \\ \hline
$1$ & $\gr\Delta'(6)\langle 1\rangle$ & $\gr\Delta'(6)\langle 1\rangle $      & $\gr\Delta'(5)\langle 1\rangle\oplus \gr\Delta'(4)\langle 1\rangle$ &         &        \\ \hline
$2$ & $\gr\Delta'(4)\langle 2\rangle$ & $\gr\Delta'(5)\langle 2\rangle$   &  &  &                        \\ \hline
\end{tabular}
\end{center}

\medskip
Using the standard Koszulity of $A'$ and the criterion \cite[Thm. 4.17 ]{PS10}, it can be shown that $\gr \wS(5,5)$ is a graded integral quasi-hereditary 
algebra, in the sense of \cite{CPS1a}.\footnote{Since $\gr A'$ is a QHA here, the criterion requires only that each module $\gr\wDelta(\lambda)$
have an irreducible head for each $\lambda\in\Lambda$. Equivalently, it much be shown that
each  $\wgr\Delta(\lambda)$ has an irreducible head.  This is also equivalent to the surjectivity of the natural
map $\wgr P(\lambda)\to\wgr\Delta(\lambda)$, which is equivalent to the term-by-term surjectivity of each of the
filtration terms used in forming these graded modules. Surjectivity of the 0th and 1st filtration term is automatic, and this fact alone gives a simple had for $\wgr\Delta(2)$, $\wgr\Delta(7)$, and $\wgr\Delta(6)$. Surjectivity for $\wgr\Delta(4)$
is automatic, since $P(4)=\Delta(4)$. The relevant filtrations for $P(5)$ and $\Delta(5)$, the last case, can be analyzed using the splitting $\wP(5)_K\cong\Delta'(4)\oplus\Delta'(5)$.
 } It can also be
shown that $\wgr A$ has an anti-involution inherited from an integral form of $A$ and which preserves
grades. Moreover, composition with the usual linear dual functor, induces a duality $X\mapsto X^\diamond$
of $\Agrmod$ and $\wgr A$-mod, which irreducible modules of pure grade 0, and sending
$L\langle r\rangle$ to $L\langle -r\rangle$. If $X$ is any $A$-module, let $\wgr^\diamond X$ denote the
graded module $(\wgr (X^\diamond))^\diamond$. Thus, $\wgr^\diamond
\nabla(\lambda)$ is obtained
by a dializing $\wgr\Delta(\lambda)=\wgr(\nabla(\lambda)^\diamond)$. Many of the tables given above
and below have natural duals which we use without comment.

Now we
can base change to $k$, to see that $\wgr S(5,5)$ is a graded quasi-hereditary algebra, with
graded PIMS having $\wgr\Delta$-filtrations described by the table below.

\medskip
\begin{center}
\begin{tabular}{|c||c|c|c|c|c|}\hline
& $\wgr P(7)$ & $\wgr P(2)$ & $\wgr P(6)$ & $\wgr P(5)$ & $\wgr P(4)$\\
\hline\hline
0 & $\begin{smallmatrix}\wgr \Delta(7)\\ \wgr\Delta(2)\end{smallmatrix}$ & $\wgr\Delta(2)$ & $\wgr\Delta(6)$ & $\begin{smallmatrix}\wgr\Delta(5)\\ \wgr\Delta(4)\end{smallmatrix}$ & $\wgr \Delta(4)$ \\
\hline 1 & $\wgr \Delta(6)\langle 1\rangle\oplus \wgr \Delta(6)\langle 1\rangle$ & $\wgr \Delta(6)\langle 1\rangle$ & $\begin{smallmatrix}\wgr \Delta(5)\langle 1\rangle \\ \wgr \Delta(4)\langle 1\rangle\end{smallmatrix}$ &&\\
\hline
2 & $\begin{smallmatrix}\wgr\Delta(5)\langle 2\rangle\\ \wgr\Delta(4)\langle 2\rangle\end{smallmatrix}$ & $\wgr\Delta(5)\langle 2\rangle$ &&& \\
\hline\end{tabular}\end{center}

We can also obtain the following resolutions:

$$\begin{cases} 0\to \wgr P(4)\to\wgr\Delta(4)\to 0\\
0\to\wgr P(4) \to \wgr P(5)\to \wgr\Delta(5)\to 0\\
0\to\wgr P(5)\langle 1\rangle \to \wgr P(6)\to\wgr\Delta(6) \to 0\\
0\to \wgr P(4)\langle 2\rangle \to \wgr P(6)\langle 1\rangle \to \wgr P(2)\to\wgr \Delta(2)\to 0\\
0\to \wgr P(4)\langle 2\rangle \to \wgr P(5)\langle 2\rangle\to \wgr P(6)\langle 1\rangle \oplus \wgr P(2)\to
\wgr P(7)\to
\wgr \Delta(7)\to 0.
\end{cases}$$

These are obtained from the above table, together with examination of various spaces
$\grHom(\wgr P(\lambda),\wgr^\diamond\nabla(\mu)\langle s\rangle)$ for various $\lambda,\mu\in
\Lambda$.\footnote{Here $\wgr^\diamond\nabla(\mu)$ is the co-standard module for $\wgr A$
associated to $\mu$.} For example, consider the more detailed structure of $\wgr P(7)$ which
the aim of providing a graded projective cover of the kernel $M$ of the
homomorphism $\wgr P(7)\twoheadrightarrow \wgr\Delta(7)$. From the table $\wgr\Delta(4)\langle 2\rangle$ appears once in a graded $\wgr\Delta$-filtration of $\wgr P(\lambda)$ appearing as a submodule.  Consequently,
 $$\dim\grHom(\wgr P(7),\wgr^\diamond\nabla(4)\langle 2\rangle)=1.$$
 The image of a non-zero element $f\in\grHom(\wgr P(7),\wgr^\diamond\nabla(4)\langle 2\rangle)$ must clearly be all
 of $\wgr^\diamond\nabla(4)\langle 2\rangle$, since the latter has head $L(7)$. Restricting $f$ to
 $M$ picks out a filtered submodule
 $$ \boxed{\begin{matrix} \wgr\Delta(6)\langle 1\rangle\\
 \wgr\Delta(5)\langle 2\rangle \\
 \wgr\Delta(4)\langle 2\rangle
 \end{matrix}}$$
 with image 
 $$\rad\wgr^\diamond\nabla(4)\langle 2\rangle =\boxed{\begin{matrix} L(6)\langle 1\rangle\\ L(5)\langle
  2\rangle\\
 L(4)\langle 2\rangle
 \end{matrix}}$$
 It follows easily that the left hand module is indecomposable with a simple head $L(6)\langle 1\rangle$.
 Consequently, it must be isomorphic to $\wgr P(6)\langle 1\rangle$. The quotient of $M$ by this
 submodule has a filtration 
 $$\boxed{\begin{matrix} \wgr P(7) \\ \wgr \Delta(2)\\ \wgr \Delta(6)\langle 1\rangle.\end{matrix}}$$
 The space $\grHom(\wgr P(7),\wgr\nabla^\diamond(6)\langle 1\rangle)$ is 2-dimensional. One of its
 basis elements has already been ``used" in the embedding $\wgr P(6)\langle 1\rangle/N\subseteq \wgr P(7)/N$
 with $N=\begin{smallmatrix}\wgr\Delta(5)\langle 2\rangle \\ \wgr\Delta(4)\langle 2\rangle\end{smallmatrix}$
 Any second basis element of this hom-space must have image 
 $$\boxed{\begin{matrix} L(7)\\ L(2)\\  L(7)\\
 L(6)\langle 1\rangle\end{matrix}}$$
 It follows $M/\wgr P(6)\langle 1\rangle$ is a homomorphic image of $\wgr P(2)$. Considering filtration
 multiplicities, the ($\wgr \Delta$-filtered) kernel of $\wgr P(6)\langle 1\rangle \to \wgr P(2)\to M\to 0$
 must be $\wgr\Delta(5)$. The rest of the resolution is easy.

\begin{thm}\label{big example} Let $A$ be the principal block of $S(5,5)$ for $p=2$. Then $\wgr A$ is standard Q-Koszul.
\end{thm}

\begin{proof}  The resolutions
above can be used to compute $\grExt^n_{\wgr A}(\wgr\Delta(\mu),\rnabla(\lambda)\langle r\rangle)$. For example,
we show  $\grExt^1_{\wgr A}(\wgr\Delta(7),\rnabla(2))=0$. Here the space $\grHom(\wgr P(2),\rnabla(2))$
of 1-cocycles  are also 1-coboundaries:
$$\begin{aligned} \grHom(\wgr P(2),\rDelta(2))&
\cong\grHom(\rDelta(2),\rnabla(2))\\
&\cong\grHom(\begin{smallmatrix} \rDelta(7)\\ \rDelta(2)\end{smallmatrix},\rnabla(2))\\
&\cong\grHom(P_0(7),\rnabla(2))\\
&\cong \grHom(\wgr P(7),\rnabla(2)).\end{aligned}
$$
The other cases are checked similarly.\end{proof}

\begin{rems} (a) Any dominant weight $\lambda$ for a semsimple, simply connected algebraic group $G$ can
be uniquely written $\lambda=\lambda_0+p\lambda$, where  $\lambda_0$ in $p$-restricted and $\lambda_1$ is
dominant. Put $\Delta^p(\lambda):=L(\lambda_0)\otimes\Delta(\lambda_1)^{[p]}$.  In 1980, Jantzen \cite{Jan} raised question whether every standard module $\Delta(\lambda)$
for a semisimple group $G$ has a $\Delta^p$-filtration, i.~e., a filtration with sections $\Delta^p(\mu)$.
See also \cite{A1}. While $\Delta(3,2)$ (as discussed in \S5.1) does not have an $\rDelta$-filtration, it does have
a $\Delta^p$-filtration. However, we know of no analogue of Q-Koszul algebras involving semisimple
groups with uses the $\Delta^p$-modules in place of the $\rDelta$-modules.

(b) It is especially interesting to compare the resolutions for $\wgr\Delta(i)$ above with the corresponding
resolutions at the quantum level (i.~e., for the $\gr \Delta'(i)$). The latter resolutions can be easily be obtained from those of the
integral versions of the $\wgr\Delta$'s and base change, together with the tables for the various
$\gr P'$-modules. We give them below: 
$$
\begin{cases} 0\to \gr P'(4)\to\gr\Delta'(4)\to 0\\
0\to\gr P'(4) \to \gr P'(5)\to \gr\Delta'(5)\to 0\\
0\to\gr P'(4)\langle 1\rangle\oplus \gr P'(5)\langle 1\rangle \to \gr P'(6)\to\gr\Delta'(6) \to 0\\
0\to \gr P'(4)\langle 2\rangle \to \gr P'(6)\langle 1\rangle \to \gr P'(2)\to\gr \Delta'(2)\to 0\\
0 \to \gr P'(5)\langle 2\rangle\to \gr P'(6)\langle 1\rangle\to
\gr P'(7)\to
\gr \Delta'(7)\to 0.
\end{cases}$$
In spite of the differences with the resolutions for the $\wgr\Delta$-modules given above, the
reader may check in each case that they lead to 
\begin{equation}\label{equality}\dim\grExt^n_{\wgr A}(\wgr \Delta(\mu), \rnabla(\lambda)\langle m\rangle)=\dim
\grExt^n_{\gr A'}(\gr\Delta'(\mu),L'(\lambda)\langle m\rangle).\end{equation}
Important in verifying this is the fact that each $(\wgr P(\lambda))_0=P_0(\lambda)$ does have
a $\rDelta$-filtration.  For example, $\grHom_{\wgr A}(P(5)\langle 1\rangle, \rnabla(4)\langle1\rangle)$ is 1-dimensional,
since $(\wgr P(5))_0=P_0(5)=\begin{smallmatrix} \rDelta(5)\\ \rDelta(4)\end{smallmatrix}$. This
discussion and others like it (such as the sample calculation in the proof of Theorem \ref{big example}) go through, even though 
$\wgr P(\lambda)_m$, $m>0$, may not have a $\rDelta$-filtration.  For example, $\wgr P(2)_4$
does not have a $\rDelta$-filtration. In spite of the latter anomaly, we still have the nice equality
(\ref{equality}) and Theorem \ref{big example}. This suggests it is more important to have $\rDelta$-filtrations at the ``top."
Also,  (\ref{equality}) has influenced a conjecture in the next section.

\end{rems}

\medskip\medskip
\begin{center}{\bf Part III: Conjectures}\end{center}
\medskip

\section{Some conjectures} Let $R$ be a classical finite root system, which we temporarily assume
 is irreducible. Let $D=1$ (respectively, 2; 3) if $R$ has type 
$A_n, D_n, E_6, E_7, E_8$ (respectively, $B_n, C_n, F_4$; $G_2$). Let $\rho=\frac{1}{2}\sum_{\alpha\in R^+}
\alpha$ be the Weyl weight. Define $g=(\rho,\theta^\vee_r)$, where $\theta_r$ is the maximal root in $R^+$.
Let $\widehat{\mathfrak g}$ be the (infinite dimensional) untwisted affine Lie algebra associated to $R$, and let
$$\widetilde{\mathfrak g} =[{\widehat{\mathfrak g}},\widehat{\mathfrak g}]=\left({\mathbb C}[t,t^{-1}]\otimes{\mathfrak g}\right)\oplus 
{\mathbb C}c$$
be its commutator subalgebra. For any $\kappa\in\mathbb Q$, consider the category $\sO_\kappa$ of $\widetilde{\mathfrak g}$-modules
satisfying certain natural properties (especially, that the central element $c$ acts as multiplication by $k$). We do not list these here, but refer instead to \cite[p. 270]{T}.

Kazhdan-Lusztig have defined, for a positive integer $\ell$, 
a functor 
\begin{equation}\label{KLcorrespondence} F_\ell:\sO_{-(\ell/2D)-g}\longrightarrow {\mathcal Q}_\ell.\end{equation}
Here ${\mathcal Q}_\ell$ is the category of integrable, type 1 modules for the Lusztig quantum
enveloping algebra corresponding to $R$ at a primitive $\ell$th root of 1. The functor $F_\ell$ is discussed
in \cite{T}, whose treatment we largely follow.
We will
be interested in the case in which $\ell$ is associated to a prime integer $p$ by the formula
\begin{equation}\label{LP}\ell=\ell(p)=\begin{cases} p
 \quad{\text{\rm  if $p$ is odd}};\\ 4\quad{\text{\rm  if $p=2$.}}\end{cases}\end{equation}
 
\medskip\noindent
\begin{defn} If $R$ is an irreducible root system, a prime $p$ is KL-good provided $F_\ell$ is an equivalence for $\ell=\ell(p)$.  Also, we assume $p\not=2$ if $R$ has type $B_n$, $C_n,$ or $F_4$, and $p\not=3$ is if $R$
has type $G_2$. More generally, if $R$ is any finite root system, then a prime $p$ is KL-good if it is KL-good for
each irreducible component of $R$.   \end{defn}

\medskip
If $R$ is of type $A_n$, every prime is KL-good. If $R$ has type $D_n$, then every odd prime is KL-good, and
$p=2$ is also KL-good if $n$ is even.  Finally, in any type, if $p>h$, the Coxeter number, then $p$ is KL-good.  See \cite[Rem. 7.3]{T} for more details.\footnote{The only cases in which $p=2$ is known to be KL-good are type $A_n$ and type $D_{2n}$.   All odd primes are known to be KL-good for simply laced classical root systems ($A_n$ and $D_n$). It would be
good to know if this fact remains true in the non-simply laced classical cases ($B_n$ and $C_n$), or at least have a bound independent of the root system.} 

We will not make use of the Kazhdan-Lusztig correspondence $F_\ell$ in the discussion below. However, it
motivates the restrictions on $p$ in the conjectures that follow. We will discuss this motivation later in this
section after Conjecture IIb. In what follows, we make use of the notation introduced at the end of Section 3.

\medskip
\noindent
{\bf Conjecture I:} {\it Assume that $G$ is a semisimple, simply connected  algebraic group, defined and split over ${\mathbb F}_p$, $p$ a prime.   Assume $p$ is KL-good for the root system $R$ of $G$. Let $\Gamma$ be a
finite ideal in $X(T)_+$, and form the QHA algebra $A:=A_\Gamma$. Then the
graded algebra $\wgr A$ is standard Q-Koszul.}
\medskip

We also expect that the $\sO$-order $\gr \wA=\gr\wA_\Gamma$ is integral quasi-hereditary (as defined in 
\cite{CPS1a}). In fact, this is likely  to be a
key
step in showing that $\wgr A=(\gr \wA)_k$ is quasi-hereditary, an essential ingredient for the standard Q-Koszul
property. If $p\geq 2h-2$ and if $\Gamma$
consists of $p$-regular blocks, then \cite{PS10} establishes that $\gr\wA$ is integral quasi-hereditary. 

We also mention that when $p\geq 2h-2$ is odd and when the Lusztig character formula holds, then Conjecture I is proved in \cite[Thm. 3.7]{PS13a} in the $p$-regular weight case. The conjecture itself has no such restrictions. In
fact, Section 4 proves the conjecture for $A$ equal to the Schur algebra $S(5,5)$ when $p=2$. All applicable conjectures
in this section have been similarly checked in that case, based on the results developed in \S6, through full proofs have not always been included there.

\medskip\noindent
{\bf Conjecture II:} {\it Continue to assume the hypotheses and notation of Conjecture I (so, in particular, $p$ is KL-good). Let $A'=A'_\Gamma$ be
the quasi-hereditary quotient algebra of the quantum enveloping algebra $U_\zeta$ at a primitive
$\ell(p)$th-root of unity associated to the ideal $\Gamma$. Then ,
$$\begin{cases}(1)\quad\dim\Ext^n_{A'}(\Delta'(\lambda),L'(\mu))&=\dim\Ext^n_A(\Delta(\lambda),\rnabla(\mu)),\\
(2)\quad \dim\Ext^n_{A'}(L'(\lambda),\nabla'(\mu)) & =\dim\Ext^n_A(\rDelta(\lambda),\nabla(\mu))\\
(3)\quad \dim\Ext^n_{A'}(L'(\lambda),L'(\mu)) & =\dim\Ext^n_A(\rDelta(\lambda),\rnabla(\mu))\end{cases},
\quad \forall \lambda,\mu\in\Lambda, \forall n$$}

\medskip
In the above expressions, the terms $\Ext^n_{A'}$ (respectively, $\Ext^n_A$) on the left (respectively, right) can be
replaced by $\Ext^n_{U_\zeta}$ (respectively, $\Ext^n_G$); see \cite{DS}, for example. When $p>h$ and the Lusztig
character formula holds for restricted dominant weights, then Conjecture II is proved for $p$-regular weights in \cite[Thm. 5.4]{CPS7}.
Some interesting cases of Conjecture II can be similarly proved assuming only $p>h$: Conjecture II(3) holds for all weights $\lambda,\mu\in pX(T)_+$.  Conjecture II(1) (respectively,
Conjecture II(2)) holds for all $\mu\in pX(T)_+$  (respectively, $\lambda\in pX(T)_+$ and all $\lambda$ (respectively, $\mu$) provided only that $p>h$. Note that $\lambda$ and $\mu$ are $p$-regular in this case. This follows from \cite[Thm. 5.4 and \S4]{CPS7}.

Related to this conjecture are the following two conjectures:

\smallskip
\noindent{\bf Conjecture IIa} {\it Under the hypothesis of Conjecture II, we have
$$\begin{cases}(1)\quad \dim\Ext^n_A(\Delta(\lambda),\rnabla(\mu))&=\dim\Ext^n_{\wgr A}(\wgr \Delta(\lambda),\rnabla(\mu))\\
(2)\quad
\dim\Ext^n_A(\rDelta(\mu),\nabla(\lambda)) & =\dim\Ext^n_{\wgr A}(\rDelta(\mu),\wgr^\diamond\nabla(\lambda))\\
(3)\quad\dim\Ext^n_A(\rDelta(\mu),\rnabla(\lambda)) & =\dim\Ext^n_{\wgr A}(\rDelta(\mu),\rnabla(\lambda)).\end{cases}$$
for all $\lambda,\mu\in \Gamma$.}

In part (2) above, $\wgr^\diamond\nabla(\lambda)$ denotes the costandard module corresponding to $\lambda$
in the highest weight category $\wgr A$. It has a natural graded structure, concentrated in non-positive grades, with
$\rnabla(\lambda)$ its grade 0 term.   Under the assumptions that $\lambda,\mu$ are $p$-regular and $p\geq 2h-2$ is
an odd prime, parts (1) and (2) of Conjecture II(a) are proved in \cite[Thm. 6.6]{PS13}, while part (3) is proved in
\cite[Thm. 5.3(b)]{PS13}.

\medskip
\noindent{\bf Conjecture IIb}: {\it Under the hypothesis of Conjecture II, we have
$$\begin{cases}(1)\quad\dim\Ext^n_{A'}(\Delta'(\lambda),L'(\mu))&=\dim\Ext^n_{\gr A'}(\gr \Delta'(\lambda),L'(\mu))\\
(2)\quad\dim\Ext^n_{A'}(L'(\lambda),\nabla'(\mu))&=\dim\Ext^n_{\gr A'}(L'(\lambda),\gr^\diamond\nabla'(\mu)),\\
(3)\quad\dim\Ext^n_{A'}(L'(\lambda),L'(\mu))&=\dim\Ext^n_{\gr A'}(\gr L'(\lambda),L'(\mu))\end{cases}$$
for all $\lambda,\mu\in \Gamma$ and all $n\in\mathbb N$. Also, $\gr A'$ is a standard Koszul algebra. }

\medskip
In fact, it also can be conjectured that the algebra $A'$ itself is standard Koszul.\footnote{A standard Koszul algebra $A$
is a Koszul algebra which is QHA, Koszul, and such that the standard (respectively, costandard) modules linear ``linear" (respectively,
``colinear"). In other words,   $A$ is a standard Q-Koszul algebra in which $A_0$ is semisimple.} In that case, Conjecture IIb would
follow immediately. In type $A_n$, this has been proved in \cite{SVV}, using, among other things, the
Kazhdan-Lusztig correspondence (\ref{KLcorrespondence}). It seems likely that these methods should extend
to all types, as long as the Kazhdan-Lusztig correspondence is an equivalence---in particular, when $p$ is KL-good.

Regarding Conjecture IIb as stated, the authors have proved that $\gr A'$ is standard Koszul in the $p$-regular case
in \cite{PS9} when $p>h$.\footnote{In fact the results of \cite{PS9}, in conjunction with the Ext-formulas
given in \cite{CPS1} that hold in the presence of ``Kazhdan-Lusztig theories," are sufficient to establish all parts of Conjecture IIb in the $p$-regular case, assuming that $p>h$.} Some weaker results, describing semisimple filtrations of standard modules, were proved for 
$p$-singular weights in \cite{PS9a}, sometimes for small $p$, provided the Kazhdan-Lusztig functor $F_{\ell(p)}$
is an equivalence.
\medskip

The next conjecture gives new calculations of Ext-group dimensions in singular weight cases.
Let ${\mathbb E}={\mathbb R}\otimes X(T)$ be the Euclidean space associated to the affine Weyl group
$W_{p}$. Thus, for $\alpha\in R$, $r\in \mathbb Z$, $s_{\alpha,rp}:{\mathbb E}\to\mathbb E$ is the reflection
defined by $s_{\alpha,rp}(x)=x-[(x,\alpha^\vee)t-pr]\alpha.$  Then $W_p$ is generated by the $s_{\alpha,pr}$, and
it is, in fact, a Coxeter group with simple reflections $S:=\{s_\alpha\equiv s_{\alpha,0}\}_{\alpha\in S} \cup \{s_{\alpha_0,-p}\}$.

 Let 
$$\overline{C^-}=\{x\in {\mathbb E}\,|\, 1\leq (x+\rho,\alpha^\vee)\leq p\}$$
be the  closed ``anti-dominant" alcove for $W_p$. For $\lambda\in X(T)$, there exists a unique $\lambda^-\in
\overline{C^-}$ which is $W_p$-conjugate to $\lambda$ under the dot action of $W_p$ on ${\mathbb E}$. Define $\overline{w}\in W_p$ to be the unique element of shortest length in $W_p$ such that $\lambda=\bar w\cdot \lambda^-$.
Alternatively, if $W_I$ denotes the stabilizer (under the dot action) in $W_p$ of an element $\lambda^-\in\overline{C^-}$, and if $w\in W_p$, then $\bar w$ denotes the distinguished (i.~e., smallest length) left coset representative for the
left coset $wW_I$.
 
Let $\sZ:={\mathbb Z}[t,t^{-1}]$ be the algebra of Laurent polynomials. Let $f\mapsto \overline f$ be the automorphism of
$\sZ$ which sends $t$ to $t^{-1}$.  Given $x,y\in W_p$, let $P_{x,y}\in\sZ$ be the Kazhdan-Lusztig polynomial assoicated to the pair
$x,y$. It is known that $P_{x,y}$ is a polynomial in $q:=t^2$.  Also, define
$$P_{\bar y,\bar w}^{\text{\rm sing}}(t):=\sum_{x\in W_I,\bar yx\leq \bar w}(-1)^{\ell(x)}P_{\bar yx,\bar w}(t),$$
 As with $P_{x,y}$, $P_{\bar y,\bar w}^{\text{\rm sing}}(t)$ is also a polynomial
in $t^2$. We have the following conjecture. We continue to assume the hypotheses and notation of Conjecture
II. In partiular, $A'$ is a quotient of $U_\zeta$, where $\zeta$ is an $\ell=\ell(p)$th primitive root of unity with $p$
KL-good. The conjecture likely holds as well for other values of $\ell$,  not associated to any prime---possibly all $\ell\geq 1$
if $R$ is simply laced. See \cite[Conj. 2.3]{L90} which conjectures some version of the Kazhdan-Lusztig correspondence works in these cases.

\medskip\noindent{\bf Conjecture III:}  Let $\Gamma$ be a finite poset in $X(T)_+$. Write $\lambda=\bar w\cdot\lambda^-$ as above. 
If $\mu\in\Gamma$ has the form $\bar y\cdot\lambda^-$ for some distinguished left coset representative $\bar y$
of $W_I$, then
$$\sum_{n\geq 0}\dim\Ext^n_{A'}(\Delta'(\mu),L'(\lambda))t^n\!= t^{\ell(\bar w)-\ell(\bar y)}\bar P_{\bar y,\bar w}^{\text{\rm sing}}(t).$$
(If $\mu$ does not have the form $\bar y\cdot\lambda^-$, then all the groups $\Ext^n_{A'}(\Delta'(\mu),L'(\lambda))=0$, as is well known.)

\medskip
In the above expression, we could replace $A'$ by $U_\zeta$, where $\zeta$ is a primitive $\ell(p)$ of unity.
The polynomials $P^{\text{\rm sing}}_{\bar y,\bar w}$ identify with one class of ``parabolic Kazhdan-Lusztig
polynomials," as introduced by Deodhar \cite{Deo} with a different notation. For further discussion, with a
different application and yet another notation, see \cite[Prop. 2.4, Cor. 4.1]{KT2}. In particular, it is shown these
polynomials always have non-negative coefficients.

 The validity of Conjecture III would explicitly calculate the dimensions of $\Ext$-groups between standard modules
 and irreducible modules  for $A'$. Then the validity of Conjecture II turns this into a similar calculation for the
 algebraic group $G$ in positive characteristic.  
  
 Conjective III also has consequences for $\Ext_{A'}$ between irreducible modules, assuming that Conjectures IIb also holds. Thus, $\gr A'$ is then standard Koszul, and 
 one obtains a product formula like that in (\ref{productformula}) at the
 level of the QHA $\gr A'$. Using Conjecture IIb again gives a calculation of $\Ext^\bullet_{A'}$-groups between irreducible $A'$-modules, once the dimensions of the groups $\Ext^a_{A'}(\Delta'(\nu),L'(\mu))$ and $\Ext^b_{A'}(L'(\lambda),\nabla'(\nu))$ can be determined. But Conjecture III calculates these dimensions in terms of Kazhdan-Lusztig coefficients. As in type $A$,
 one might expect that $A'$ itself is standard Koszul (which would simplify the above discussion).
  
 Once the $\Ext_{A'}$-groups in the above paragraph have been calculated, the corresponding $\Ext^\bullet_A=\Ext^\bullet_G$-groups
 can be calculated, if Conjecture II also holds. Specifically, the dimensions of all groups
  $$\Ext_G^\bullet(\rDelta(\lambda),\rnabla(\mu)), \quad
 \Ext_G^\bullet(\Delta(\lambda),\rnabla(\mu)), \quad \Ext_G^\bullet(\rnabla(\lambda),\nabla(\mu))$$
 for all $\lambda,\mu\in X(T)_+$
 can be calculated explicitly in terms of Kazhdan-Lusztig polynomials and their coefficients.
  Miraculously, all these predicted dimensions are correct for $A=S(5,5)$ and $p=2$.

\begin{rems} \label{finremark} (a) 
Conjecture II is a consequence of Conjectures IIa and IIb, in the presence of Conjecture I. This is obtained by proving,
using the standard Koszulity of $\gr A'$ and the standard  Q-Koszulity of $\wgr A$, that 
\begin{equation}\label{gradedext}\begin{aligned}\dim\grExt^n_{\gr A'}(\gr\Delta'(\lambda),L'(\mu)\langle n\rangle)&=\dim\grExt^n_{\wgr A}(\wgr \Delta(\lambda),
\rnabla(\mu)\langle n\rangle)\\ & \forall \lambda,\mu\in\Lambda, n\in\mathbb N.\end{aligned}\end{equation}
This gives the first equality in Conjecture II, by passing to  Ext-groups and $A,A'$. The second and third
equalities can be proved similarly, using relevant similar versions of (\ref{gradedext}).

Conjecture III is suggested by using $\grExt$-groups in an affine Lie algebra setting, assuming Koszulity. It would
then follow from the existence there of well-behaved graded translation functors, passing from $p$-regular to $p$-singular
graded module categories.

(b) One cumulative effect of all the conjectures is to explicitly compute $\Ext$-groups for $A$ and $\wgr A$ between
objects $\Delta^0(\lambda)$ and $\nabla_0(\mu)$. One can ask if there is any resulting impact on calculations of
similar Ext-groups between irreducible modules. We speculate that the homological algebra of $\wgr A$, e.~g., $\Ext$
between irreducible modules or between irreducible and standard/costandard modules can often be understood in terms of the similar homological algebra
of $(\wgr A)_0$, together with formulas like (\ref{productformula}). Observe that $(\wgr A)_0$
is a quotient of $A$ by a nilpotent ideal, so that $A$, $(\wgr A)_0$, and $\wgr A$ all share the same irreducible
modules. 
\end{rems}


\end{document}